\documentclass[11pt]{amsart}

\usepackage{amssymb,amsmath,amscd,graphicx,
latexsym,amsthm}
\usepackage{amssymb,latexsym,amsmath,amscd,graphicx}
\usepackage[mathscr]{eucal}
  \usepackage[all]{xy}
        \def\Q{{\mathcal Q}}
\def\CC{{\mathbb C}}                  
\def\PP{{\mathbb P}}
\def\OO{{\mathcal O}}
\def\R{{\mathbf R}}

\def\F{\mathcal{F}}
\def\E{\mathcal{E}}
\def\G{\mathcal{G}}

\def\I{\mathcal{I}}

\def\cP{\mathcal{P}}
\def\Pic0{{\rm Pic}^0}

\def\DD{{\mathbf{D}}}
\def\R{{\mathbf{R}}}

\def\MM{\mathcal M}
\def\T{\mathcal T}

\def\Ext{\mathrm{Ext}}
\def\Hom{\mathrm{Hom}}
\def\EExt{\mathcal{E}xt}

\def\TP{{\widetilde \cP}}

\def\N{\mathcal N}

\def\on{{\otimes n}}
\def\cR{{\mathcal R}}
\def\DDelta{{\Delta_Y}}

\theoremstyle{plain}

\newtheorem{theorem}{Theorem}[section]

\newtheorem{proposition/example}[theorem]{Proposition/Example}
\newtheorem{proposition}[theorem]{Proposition}

\newtheorem{corollary}[theorem]{Corollary}
\newtheorem{lemma}[theorem]{Lemma}

\newtheorem{claim}[theorem]{Claim}

\theoremstyle{definition}
\newtheorem{definition}[theorem]{Definition}

\newtheorem{notation}[theorem]{Notation}
\newtheorem{remark}[theorem]{Remark}

\newtheorem{conjecture/question}[theorem]{Conjecture/Question}

\newtheorem{remark/definition}[theorem]{Remark/Definition}
\newtheorem{notation/assumptions}[theorem]{Assumptions/Notation}

\numberwithin{equation}{section}

 \theoremstyle{remark}

\begin{document}

\title{Gaussian maps and generic vanishing I: \\  subvarieties of abelian varieties}

\author[G. Pareschi]{Giuseppe Pareschi}
\address{Dipartamento di Matematica, Universit\`a di Roma, Tor Vergata, V.le della
Ricerca Scientifica, I-00133 Roma, Italy} \email{{\tt
pareschi@mat.uniroma2.it}}
\dedicatory{Dedicated to my teacher, Rob Lazarsfeld, on the occasion of his $60$th birthday}

\thanks{}

\date{\today}
\maketitle


\setlength{\parskip}{.1 in}

\markboth{G. PARESCHI} {Gaussian maps and generic vanishing I:  subvarieties of abelian varieties }

\begin{abstract}
The aim of this paper is to present an approach to Green-Lazarsfeld's generic vanishing combining gaussian maps  and the Fourier-Mukai transform associated to the Poincar\`e line bundle. As an application we prove the Generic Vanishing Theorem for all normal Cohen-Macaulay subvarieties of abelian varieties over an algebraically closed field. 
\end{abstract}

\section{Introduction}

We work with irreducible projective varieties on an algebraically closed field of any characteristic, henceforth called \emph{varieties}. The contents of this paper are:

\noindent (1)  a general criterion expressing the vanishing of the higher cohomology of a line bundle on a Cohen-Macaulay variety in terms of a certain first-order conditions on hyperplane sections (Theorem \ref{gaussian}). Such conditions involve  \emph{gaussian maps} and the criterion is a  generalization of well known results on hyperplane sections of K3 and abelian surfaces; 

\noindent (2) using a relative version of the above,  we prove the vanishing of higher direct images of Poincar\'e line bundles of normal Cohen-Macaulay  subvarieties of abelian varieties\footnote{by Poincar\'e line bundle of a subvariety $X$ of an abelian variety $A$ we mean the pullback to $X\times \Pic0 A$ of a Poincar\'e line bundle on $A\times \Pic0 A$} (Theorem \ref{subvarieties}).  As it is well known, this is equivalent to Green-Lazarsfeld's \emph{generic vanishing}, a condition satisfied by all irregular  compact Kahler manifolds (\cite{gl1}). This implies in turn a Kodaira-type vanishing  for  line bundles which are restrictions to normal Cohen-Macaulay subvarieties of abelian varieties, of ample line bundles on the abelian variety (Corollary \ref{kodaira}).

Concerning point (2), it should be mentioned that at present we are not able to extend efficiently this approach to the general Generic Vanishing Theorem (GVT),  i.e. for varieties \emph{mapping to} abelian varieties, even for smooth projective varieties over the complex numbers (where it is well known by the work of Green and Lazarsfeld). This will be the object of further research. However, concerning possible extensions of the general GVT  to singular varieties and/or to positive characteristic
one should keep in mind the work of Hacon and Kovacs \cite{hacko} where -- by exploiting the relation between GVT and the Grauert-Riemenschneider vanishing theorem --   they show examples of failure of the GVT for mildly singular varieties  (over $\CC$) and even smooth varieties (in characteristic $p>0$) of dimension $\ge 3$, with a (separable) generically finite map to an abelian variety. This disproved an erroneous theorem of a previous preprint of the author.


Now we turn to a more detailed presentation of the above topics.

\noindent \subsection{Motivation: gaussian maps on curves and vanishing of the {${\mathbf H^1}$} of line bundles on surfaces. }\label{motivation} We introduce part (1) starting from a particular case, where the essence of story becomes apparent: the vanishing of the $H^1$ of a line bundle on a surface in terms of gaussian maps on a sufficiently positive hyperplane section (Theorem \ref{intro} below). 

To begin with, let us recall what  gaussian maps are. Given a curve $C$ and a line bundle $A$ on $C$, denote $M_A$ the kernel of the evaluation map of global sections of $A$:
\[ 0\rightarrow M_A\rightarrow H^0(C,A)\otimes\OO_C\rightarrow A\]
 It comes equipped with a natural $\OO_C$-linear differentiation map
\[M_A\rightarrow \Omega^1_C\otimes A\]
 defined as 
\[M_A=p_*(\I_\Delta\otimes q^*A)\rightarrow p_*((\I_\Delta\otimes A)_{|\Delta})=\Omega^1_C\otimes A\]
where $p$, $q$ and $\Delta$ are the projections and the diagonal of the product $C\times C$.
Twisting with another line bundle $B$ and taking global sections one gets the \emph{gaussian map} (or \emph{Wahl map})  of $A$ and $B$:
\[\gamma_{A,B}: Rel(A,B):=H^0(C,M_A\otimes B)\rightarrow H^0(C,\Omega^1_C\otimes A\otimes B)\>\footnote
{the source is denoted $Rel(A,B)$, as it is the kernel of the multiplication of global sections of $A$ and $B$}\]
In our treatment it is more natural to set $A=N\otimes P$ and $B=\omega_C\otimes P^\vee$ for suitable line bundles $N$ and $P$ on the curve $C$, and to consider the dual  map
\begin{equation}\label{g}
g_{N,P}:\Ext^1_C(\Omega^1_C\otimes N,\OO_C)\rightarrow \Ext^1_C(M_{N\otimes P},P)
\end{equation}
Note that $g_{N,P}$ can defined directly (even if $\omega_C$ is not a line bundle) as \break 
$\Ext^1_C(\,\cdot\,,P)$ of the differentiation map of $M_{N\otimes P}$. 

The relation with  the vanishing of the $H^1$ of line bundles on surfaces is in the following result, whose proof follows closely arguments contained in the papers of Beaville and M\'erindol \cite{bm} and Colombo, Frediani and the author \cite{cfp}. Let $X$ be a Cohen-Macaulay surface and let $Q$ a line bundle on $X$. Let $L$ a base point-free line bundle on $X$ such that also $L\otimes Q$ is base point-free, and let $C$ be a (reduced and irreducible) Cartier divisor in $|L|$, not contained in the singular locus of $X$. Let $N_C=L_{|C}$ be the normal bundle of $C$. We have the extension class $$e\in \Ext^1_C(\Omega^1_C\otimes N_C,\OO_C)$$ of the normal sequence
\[0\rightarrow N_C^\vee\rightarrow (\Omega^1_X)_{|C}\rightarrow \Omega^1_C\rightarrow 0\]
 We consider the (dual) gaussian map
\begin{equation}\label{specialized}
g_{N_C,Q_{|C}}:\Ext^1_C(\Omega^1_C\otimes N_C,\OO_C)\rightarrow \Ext^1_C(M_{N_C\otimes Q},Q_{|C})
\end{equation}

\begin{theorem}\label{intro} (a) If $H^1(X,Q)=0$ then $e\in \ker (g_{N_C,Q_{|C}})$.\\
(b) If $L$ is sufficiently positive\footnote{by this  we mean that $L$ is a sufficiently high multiple of a fixed ample line bundle on $X$. } then also the converse holds: if $e\in \ker (g_{N_C,Q_{|C}})$ then $H^1(X,Q)=0$.
\end{theorem}

\noindent (note that $e$ is non-zero if $L$ sufficiently positive). For example, if $X$ is a smooth surface with trivial canonical bundle  and $Q=\OO_X$ then (a) says that if $X$ is a K3 then $e\in \ker(g_{K_C,\OO_C})$. This is a result of \cite{bm}. Conversely, (b) says that if $X$ is abelian  and $C$ is sufficiently positive then $e\not\in \ker(g_{K_C,\OO_C})$. This a result of  \cite{cfp}.

 \subsection{}\label{discussione} The proof is a calculation with extension classes whose geometric motivation is as follows. Suppose that $C$ is a curve in a surface $X$ and that $C$  is embedded in an ambient variety $Z$. From the cotangent sequence
\[0\rightarrow \I/\I^2\rightarrow (\Omega^1_Z)_{|C}\rightarrow \Omega^1_C\rightarrow 0\]
(where $\I$ is the ideal of $C$ in $Z$) one gets the long cohomology sequence
\begin{equation}\label{H-G}
\cdots\rightarrow \Hom_C(\I/\I^2,N_C^\vee)\buildrel{H_Z}\over\rightarrow \Ext^1_C(\Omega^1_C,N_C^\vee)\buildrel{G_Z}\over\rightarrow \Ext^1_C((\Omega^1_Z)_{|C},N_C^\vee)\rightarrow\cdots
\end{equation}
 The problem of extending the embedding $C\hookrightarrow Z$ to the surface $X$ has an natural first-order obstruction, namely the class $e$ must belong to $\ker G_Z={\mathrm Im}\, H_Z$. Indeed, as it is well known, if the divisor $2C$ on $X$, seen as a scheme, is embedded in $Z$,  then it lives (as embedded first-order deformation) in the $\Hom$ on the left\footnote{more precisely, the ideal of $2C$ in $Z$ induces the morphism of $\OO_Z$-modules $\I/\I^2\rightarrow N_C$ whose kernel is $\I_{2C/Z}/\I^2$, see e.g. \cite{bf} or \cite{ferrand} }. The forgetful map
$H_Z$, disregarding the embedding, takes it to  the class of the normal sequence $e\in  \Ext^1_C(\Omega^1_C,N_C^\vee)$. 

Now we specialize this to the case when the ambient variety  is a projective space, specifically: $$Z=\PP(H^0(C,N_C\otimes Q)^\vee):=\PP_Q$$  (in this informal discussion we are assuming, for simplicity, that the line bundle $L\otimes Q$ is very ample).  By the Euler sequence the map $G_{\PP_Q}$ is the (dual) gaussian map $g_{N_C,Q_{|C}}$ of (\ref{specialized}). Notice that in this case there is the special feature that our extension problem can be relaxed to the problem of extending the embedding of $C$ in $\PP_Q$ to an embedding of the surface $X$ in a possibly bigger projective space $\PP$, containing $\PP_Q$ as a linear subspace. However, since the restriction to $\PP_Q$ of the conormal sheaf of $C$ in  $\PP$ splits, this has the same first-order obstruction, namely $e\in \ker (g_{N_C,\Q_{|C}})$. 

The relation of all that with the vanishing of the $H^1$ is classical: the embedding of $C$ in $\PP_Q$ can be extended (in the above relaxed sense) to an embedding of $X$ if and only if the restriction map $\rho_X: H^0(X,L\otimes Q)\rightarrow H^0(C,N_C\otimes Q)$ is surjective. This is implied by the vanishing of $H^1(X,Q)$, so we get (a).  The converse is  a bit more complicated: by Serre vanishing, if $L$ is sufficiently positive the vanishing of $H^1(X,Q)$ is \emph{equivalent} to the surjectivity of the restriction map $\rho_X$, and also to the surjectivity of the restriction map $\rho_{2C}: H^0(2C,(L\otimes Q)_{|2C})\rightarrow H^0(C,N_C\otimes Q)$, hence to the fact that $2C$ 'lives' in $\Hom_C(\I/\I^2,N_C^\vee)$. Now if  $e$ is in the kernel of $g_{N_C,Q_{|C}}=G_{\PP_Q}$ then $e$ comes from some embedded deformations in $\Hom_C(\I/\I^2,N_C^\vee)$. However these do not  necessarily include $2C$. A more refined analysis 
proves that this is indeed the case as soon as $L$ is sufficiently positive. 

\subsection{Gaussian maps on hyperplane sections and vanishing. }\label{subsection2} The criterion of part (1) above is a generalization of the previous Theorem to higher dimension and to a relative (flat) setting. The relevant case deals with the vanishing of the $H^n$ of line bundle  on variety of dimension $n+1$ \footnote{ In fact, for all positive $k$, with $k<n$, the vanishing of $H^k$ can be reduced to this case, as it is equivalent (by Serre vanishing) to the vanishing of $H^k$ of the restriction of the given line bundle to a sufficiently positive $k+1$-dimensional hyperplane section}
\footnote{Note: one could think to use the equality $h^n(X,Q)=h^1(X,\omega_X\otimes Q^\vee)$ and then reduce, as in the previous footnote, to a surface. However, this is not possible in the relative case, since in general there is no Serre duality isomorphism of the direct images. Even in the non-relative case  the resulting criterion is usually more difficult to apply}.  To this purpose, we  consider "hybrid" gaussian maps as follows: let $C$ be a curve in a n-dimensional variety $Y$ and let $A_C$ be a line bundle on $C$.  The \emph{Lazarsfeld sheaf}  (see \cite{laz}), denoted $F^Y_{A_C}$, is the kernel of the evaluation map   of $A_C$, \emph{seen as a sheaf on $Y$}:
 \[0\rightarrow F^Y_{A_C}\rightarrow H^0(A_C)\otimes \OO_Y\rightarrow A_C\]
(note that $F^Y_{A_C }$ is never locally free if $\dim Y\ge 3$). As above, it comes equipped with a $\OO_Y$-linear differentiation map 
 \[F^Y_{A_C}\rightarrow \Omega^1_Y\otimes A_C\]
  If $B$ is a line bundle on $Y$, we define the \emph{ gaussian map of $A_C$ and $B$} as
\[\gamma^Y_{A_C,B}: Rel(A_C,B)=H^0( Y, F_{A_C}\otimes B)\rightarrow H^0(Y, \Omega^1_Y\otimes A_C\otimes B)\]
As above, we will rather use the dual map
\[g^Y_{M_C,R}: \Ext^n_Y(\Omega^1_Y\otimes M_C ,\OO_Y)\rightarrow \Ext^n_Y(F_{M_C\otimes R},R)\]
where $M_C$ and $R$ are line bundles respectively on $C$ and $Y$ such that $A_C=M_C\otimes R$ and $B=\omega_Y\otimes R^\vee$. Again, this map 
can defined directly (even if $\omega_Y$ is not a line bundle) as $\Ext^1_C(\,\cdot\,,R)$ of the differentiation map of $F^Y_{M_C\otimes R}$.  The case $n=1$ is recovered taking $Y=C$.

These maps can be extended to a relative flat setting. In this paper we consider only the simplest case, namely a family of line bundles on a fixed variety $Y$, as this is the only one needed in the subsequent applications. In the notation above, let $T$ be a another projective CM variety (or scheme), and let $\cR$ be a line bundle  on $Y\times T$. Let $\nu$ and $\pi$  denote the two projections, respectively on $Y$ and $T$. 
 We can consider the \emph{relative Lazarsfeld sheaf} $\F^Y_{M_C,\cR}$, kernel of the relative evaluation map 
 \[0\rightarrow\F^Y_{M_C,\cR}\rightarrow\pi^*\pi_*(\cR\otimes \nu^*M_C)\rightarrow \cR\otimes \nu^*M_C\] where, as above, we see $M_C$ as a sheaf on $Y$. 
The $\OO_{Y\times T}$-module $\F^Y_{M_C,\cR}$ is  equipped with its $\OO_{Y\times T}$-linear differentiation map (see Subsection \ref{preliminaries} below)
\begin{equation}\label{diff}\F^Y_{M_C,\cR}\rightarrow \nu^*(\Omega^1_Y\otimes M_C)\otimes \cR
\end{equation}
Applying $\Ext^n_{Y\times T}( \>\cdot\>,\cR)$ and restricting to the direct summand \break
$\Ext^n_Y(\Omega^1_Y\otimes M_C,\OO_Y)$ we get the (dual) gaussian map
$$g^Y_{M_C,\cR}: \Ext^n_Y(\Omega^1_Y\otimes M_C,\OO_Y)\rightarrow \Ext^n_{Y\times T}(\F^Y_{M_C,\cR}\, ,\cR)$$

The announced generalization of the  Theorem \ref{intro} is as follows. Let $X$ be a $n+1$-dimensional Cohen-Macaulay variety, let $T$ be a CM variety, and let $\Q$ be a line bundle on $X\times T$. In order to avoid heavy notation, we still denote $\nu$ and $\pi$ the two projections of $X\times T$ (however see Notation \ref{projections} below).
Let  $L$ be a line bundle on $X$,
with $n$ irreducible effective divisors $Y_1,\dots , Y_n\in |L|$ such that their intersection is an integral curve $C$  not contained in the singular locus of $X$. We assume also that the line bundle $\Q\otimes\nu^*L^\on$ is relatively base point-free, namely the relative evaluation map $\pi^*\pi_*(\Q\otimes\nu^*L^\on)\rightarrow \Q\otimes \nu^*L^\on$ is surjective. 
 We choose a divisor among $Y_1,\dots ,Y_n$, say $Y=Y_1$, such that $C$ is not contained in the singular locus of $Y$.  Let $N_C$ denote the line bundle $L_{|C}$. We consider the "restricted normal sequence" 
\begin{equation}\label{restricted normal}0\rightarrow N_C^\vee\rightarrow (\Omega^1_X)_{|C}\rightarrow (\Omega^1_Y)_{|C}\rightarrow 0
\end{equation}
Via the canonical isomorphism
\begin{equation}\label{koszul}\Ext^1_C(\Omega^1_Y\otimes N_C,\OO_C)\cong 
\Ext^n_Y(\Omega^1_Y\otimes N^{\otimes n}_C,\OO_Y)
\end{equation}
(see Subsection \ref{preliminaries} below)
we can see the class $e$ of (\ref{restricted normal})
as belonging to $\Ext^n_Y(\Omega^1_Y\otimes N^{\otimes n}_C,\OO_Y)$. Finally, we consider the (dual) gaussian map
\begin{equation}\label{specialized2}
g_{N_C^\on,\Q_{|Y\times T}}:\Ext^n_Y(\Omega^1_Y\otimes N_C^\on,\OO_Y)\rightarrow \Ext^n_{Y\times T}(\F^Y_{N_C^\on,\Q_{|Y\times T}}\, ,\Q_{|Y\times T})
\end{equation}
Then we have the following result, recovering part (b) the Theorem \ref{intro} as  the case $n=1$ and $T=\{point\}$ 

\begin{theorem}\label{gaussian}
 If $L$ is sufficiently positive and $e\in \ker (g^Y_{N^{\otimes n}_C,\Q_{|Y\times T}})$ then $R^n\pi_*\Q=0$.
\end{theorem}

The following version is technically easier to apply

\begin{corollary}\label{B} Keeping the notation of Theorem \ref{gaussian},
if the line bundle $L$ is sufficiently positive then the kernel of the  map ${g}^Y_{N^{\otimes n}_{C},\Q_{|Y\times T}}$ is at most 1-dimensional (spanned by $e$). Therefore if  ${g}^Y_{N^{\otimes n}_{C},\Q_{|Y\times T}}$ is non-injective then  $R^n\pi_*\Q=0$. 
\end{corollary}

Concerning the other implication what we can prove is
\begin{proposition}\label{Aprimo} (a) Assume that $T=\{point\}$. If $H^n(X,\Q)=0$ then $e\in
 \ker (g_{N_C^{\otimes n},\Q_{|Y}})$. \\
 (b) In  general, assume that $R^i\pi_*(\Q_{|Y\times T})=0$ for $i<n$. If  also $R^n\pi_*\Q=0$ then $e\in \ker (g^Y_{N_C^{\otimes n},\Q_{|Y\times T}})$.\\
\end{proposition}

\subsection{}\label{discussione2} To motivate these statements, let us go back to the informal discussion of Subsection   \ref{discussione}. We assume for simplicity that $T=\{point\}$. Let $X$ be a $n+1$-dimensional variety and $C$ a curve in $X$ as above. It is easily seen, using  the Koszul resolution  of the ideal of $C$ and Serre vanishing, that the vanishing of $H^n(X,Q)$ implies the surjectivity of the restriction map \break
$\rho_X: H^0(X,L^{\otimes n}\otimes Q)\rightarrow H^0(C, N_C^{\otimes n}\otimes Q)$, and in fact the two conditions are equivalent as soon as $L$ is sufficiently positive. Hence it is natural to look for first order obstructions to extending to $X$  an embedding of the curve $C$ (a $1$-dimensional complete intersection of linearly equivalent divisors of $X$) into 
$$\PP_Q:=\PP(H^0(C,N_C^\on\otimes Q)^\vee)\>\>.$$ More generally, we can consider the same problem for any given ambient variety $Z$, rather than projective space. 

  To find a first-order obstruction one cannot anymore replace $X$ with the first order neighborhood of $C$ in $X$. We rather have to pick a divisor in $|L|$ containing $C$, say $Y=Y_1$ and replace $X$ with the scheme $2Y\cap Y_2\cap \dots \cap Y_n$. 
In analogy with the case of curves on surfaces, it is  natural to consider the long cohomology sequence
\begin{equation}\label{general}
\cdots\rightarrow \Hom_(\I_Y/\I_Y^2,N_C^\vee)\buildrel{H^Y_Z}\over\rightarrow \Ext^1_C((\Omega^1_Y)_{|C},N_C^\vee)\buildrel{G^Y_Z}\over\rightarrow \Ext^1_C((\Omega^1_Z)_{|C},N_C^\vee)\rightarrow\cdots
\end{equation}
(where $\I_Y$ is the ideal of $Y$ in $Z$).
As above, a necessary condition for the lifting to $X$ of the embedding of $C\hookrightarrow Z$ is that the "restricted normal class" $e$ of (\ref{restricted normal}) belongs to $\ker (G^Y_Z)$.  

However, looking for \emph{sufficient} conditions for  lifting  (in the relaxed sense, as in Subsection \ref{discussione})   the embedding $C\hookrightarrow \PP_Q$ to $X$, 
one cannot assume that  the divisor $Y$ is already embedded in $\PP_Q$. This is the reason why, differently from the  case when $X$ is a surface,  the map $g_{N_C^{\otimes n},Q_{|Y}}$ appearing in the statement of Theorem \ref{gaussian} and Corollary \ref{B} is not the map  $G^Y_{Z}$ with $Z=\PP_Q$, but rather a slightly more complicated "hybrid" version of (dual) gaussian map. After this modification, the geometric motivation for Theorem \ref{gaussian} is similar to that of Subsection \ref{discussione}.

\subsection{Generic vanishing for subvarieties of abelian varieties}\label{subsection3} Although difficult -- if not impossible -- to use in  most cases, the above results can be applied in some very special circumstances. For example, in analogy with the literature on curves sitting on K3 surfaces and Fano 3-folds, Proposition \ref{Aprimo} can supply non-trivial necessary conditions for a $n$-dimensional variety to sit in some very special  $n+1$-dimensional varieties.

 However in this paper we rather focus on the sufficient condition for vanishing provided by Theorem \ref{gaussian} and   Corollary \ref{B}, as it provides an  approach to  \emph{generic vanishing}, a far-reaching concept introduced by Green and Lazarsfeld in the papers \cite{gl1} and \cite{gl2}.
 Namely we consider a variety $X$ with a map to an abelian variety, generically finite onto its image
 \begin{equation}\label{a}a:X\rightarrow A
 \end{equation}
 Denoting $\Pic0A=\widehat A$ the dual variety, we consider the pullback to $X\times \widehat A$ of a Poincar\'e line bundle $\cP$ on $A\times \widehat A$:
 \begin{equation}\label{poincare}
 \Q=(a \times \mathrm{id}_{\widehat A})^*\cP
 \end{equation}
 We keep the notation of the previous section. In particular we denote  $\nu$ and $\pi$ the projections of $X\times \widehat A$. A  way of expressing generic vanishing is the vanishing  of higher direct images
 \begin{equation}\label{vanishing}R^i\pi_*\Q=0\quad \hbox{for} \quad i<\dim X
 \end{equation}
 For smooth varieties over the complex numbers (\ref{vanishing}) was proved (as a particular case of a more general statement) by Hacon (\cite{hacon}), settling a conjecture of Green and Lazarsfeld. Another way of expressing the generic vanishing condition involves the \emph{cohomological support loci}  
$$V^i_a(X)=\{\alpha\in\Pic0A\>|\> h^i(X, a^*\alpha)>0\} \>\>\>.
$$ 
Green-Lazarsfeld's theorem (\cite{gl1}, \cite{gl2}) is that, if the map $a$ is generically finite, then
  \begin{equation}\label{GL}\mathrm{codim}_{\widehat A}\,V^i_a(X)\ge \dim X-i\>\> \footnote{In general, if the map $a$ is not generically finite, Hacon's and Green-Lazarsfeld's theorems are respectively  $R^i\pi_*\Q=0$ for $i<\dim a(X)$ and $\mathrm{codim}_{\Pic0A}\,V^i_a(X)\ge \dim a(X)-i$. However they can be reduced to the case of generically finite $a$ by taking sufficiently positive hyperplane sections of dimension equal to the rank of $a$} .
   \end{equation}
   It is easy to see that (\ref{vanishing}) implies (\ref{GL}). Subsequently, it has been observed in \cite{pp1} and \cite{pp2} that (\ref{vanishing}) is in fact \emph{equivalent} to (\ref{GL})\footnote{in \cite{pp2} this is stated only in the smooth case, but this hypothesis is unnecessary}.     The heart of Hacon's proof of (\ref{vanishing}) consists in a clever reduction to Kodaira-Kawamata-Viehweg vanishing, while the argument of Green and Lazarsfeld for (\ref{GL}) uses Hodge theory. Both need the characteristic zero and that the variety $X$ is smooth (or with rational singularities). 
   
   On the other hand
  a characteristic-free example of both (\ref{vanishing}) and (\ref{GL}) is given by abelian varieties themselves (Mumford, \cite{M} p.127). Here we extend this by proving  that (\ref{vanishing}) (and, therefore,  (\ref{GL})) holds  for normal Cohen-Macaulay subvarieties of abelian varieties on an algebraically closed field of any characteristic:

 \begin{theorem}\label{subvarieties}  In the above notation, assume that $X$ is normal Cohen-Macaulay and the morphism $a$ is an embedding. Then $R^i\pi_*\Q=0$ for all $i<\dim X$.
 \end{theorem}
 
 The strategy of the proof consists  in applying Theorem \ref{gaussian} to the Poincar\'e line bundle $\Q$.  In order to do so we take a  general complete intersection $C=Y\cap Y_2\cap \dots\cap Y_n$ of $X$, with $Y_i\in |L|$, where, as above, $L$ is a sufficiently positive line bundle on $X$ and $n+1=\dim X$. 
The main issue of the argument consists in comparing two spaces of first-order deformations:  the first is  the kernel of the (dual) gaussian map $g_{N_C^\on,\Q_{Y\times \widehat A}}$.  
The second is the kernel of the map $G^Y_Z$  of (\ref{general}) with $Z$ equal to the ambient abelian variety $A$ \footnote{this is simply the dual of the multiplication map
 $V\otimes H^0( N_C\otimes \omega_C)\rightarrow H^0(\Omega^1_Y\otimes N_C\otimes\omega_C)$, where $V$ is the cotangent space of $A$ at the origin} (by (\ref{koszul}) the two maps have the same source). As in the discussion of Subsection \ref{discussione2}, the variety $X\subset A$ induces naturally, via the restricted normal extension class $e$, a non-trivial element of $\ker G^Y_A$. Hence, to get the vanishing of $R^n\pi_*\Q$ it would be enough to prove that $\ker G^Y_A$ is contained in $\ker g_{N_C^\on,\Q_{Y\times \widehat A}}$, or at least -- in view of Corollary \ref{B} --  that
 the intersection of $\ker G^Y_A$ and $\ker g_{N_C^\on,\Q_{|Y\times \widehat A}}$ is non-zero. This analysis is accomplished by means of the  Fourier-Mukai transform associated to the Poincar\`e line bundle\footnote{we remark, incidentally, that  (\ref{vanishing}) for abelian varieties (Mumford's theorem) is the key point assuring that the Fourier-Mukai transform is an equivalence of categories}. In doing this we were inspired by the classical papers \cite{mu1} and \cite{kempf} where it was solved the conceptually related problem of comparing the first-order embedded deformations of a curve in its jacobian and the first-order deformations of the Picard bundle  on the dual.

  The vanishing of $R^i\pi_*\Q$ for $i<n$ follows from this step, after reducing to a sufficiently positive $i+1$-dimensional hyperplane section.
 
 Note that conditions (\ref{GL}) can be expressed dually as
 $$\mathrm{codim}_{\Pic0A}\bigl\{\alpha\in\Pic0A\>|\> h^i(\omega_X\otimes\alpha)>0\>\bigr\}\ge i\qquad\hbox{for all}\quad i>0$$
According to the terminology of \cite{pp2}, this is stated saying that the dualizing sheaf $\omega_X$ is a GV-sheaf.  As a first application of Theorem \ref{gaussian} we note that, combining with Proposition 3.1 of \cite{pp3} ("GV tensor $IT_0\rightarrow IT_0$") we get the following Kodaira-type vanishing
 
 \begin{corollary}\label{kodaira} Let $X$ be a normal Cohen-Macaulay subvariety of an abelian variety $A$, and let $L$ be an ample line bundle on $A$. Then $H^i(X,\omega_X\otimes L)=0$ for all $i>0$.
 \end{corollary}

 It seems possible that these methods can find application in wider generality.
 
The paper is organized as follows: in Section \ref{section2} we prove Theorem \ref{gaussian} (and Proposition \ref{Aprimo}).  Section \ref{section3} contains the proof of Corollary \ref{B}. In Sections \ref{section4}  and \ref{section5} we establish the set up of the argument for Theorem \ref{subvarieties}. In particular, we interpret gaussian maps in terms of the Fourier-Mukai transform.  The conclusion of the proof of Theorem \ref{subvarieties} takes up Section \ref{section6}.

\section{Proof of Theorem \ref{gaussian} and Proposition \ref{Aprimo}}\label{section2}

\subsection
{Preliminaries} \label{preliminaries} The argument consists of a computation with extension classes. The geometric motivation  is outlined in the Introduction (Subsections  \ref{discussione} and \ref{discussione2}). To get a first idea of the argument, it could be helpful to  have a look at the proof of  Lemma 3.1 of  \cite{cfp} .  

\begin{notation}\label{projections}  In the first place, some warning about the notation. We have the three varieties $C\subset Y\subset X$ (respectively of dimension $1$, $n$ and $n+1$). 
The projections of $X\times T$ onto $X$ and $T$ are denoted respectively $\nu$ and $\pi$. It will be different to consider the relative evaluation maps of a sheaf ${\mathcal A}$ on $C\times T$ seen as a sheaf  on $Y\times T$, or on $X\times T$, or on $C\times T$ itself: their kernels are the various different relative  Lazarsfeld sheaves attached to ${\mathcal A}$ in different ambient varieties (see Subsection \ref{subsection2}). Therefore we denote 
$$\pi_Y=\pi_{|Y\times T}\qquad\pi_C=\pi_{|C\times T}$$
For example,  on $Y\times T$ we have 
\begin{equation}\label{FY}
0\rightarrow \F^Y_{A,\Q_{|Y\times T}}\rightarrow \pi_{Y}^*\pi_*(\Q\otimes\nu^*A)\rightarrow \Q\otimes\nu^*A\end{equation}
while on $X\times T$
\begin{equation}\label{FX}0\rightarrow \F^X_{A,\Q}\rightarrow \pi^*\pi_*(\Q\otimes\nu^*A)\rightarrow \Q\otimes\nu^*A
\end{equation}
\end{notation}

Next, we clarify a few points appearing in the Introduction.

\medskip\noindent \textbf{The differentiation map (\ref{diff}).  } We describe explicitly the differentiation map (\ref{diff}). We keep the notation there: $M_C$ is a line bundle on the curve $C$ while $\cR$ is a line bundle on $Y\times T$. Now let $p$,$q$ and $\widetilde\Delta$ denote the two projections and the diagonal of the fibred product $(Y\times T)\times_T(Y\times T)$.
Concerning the Lazarsfeld sheaf $\F^Y_{M_C,\cR} $ we claim that there is a canonical isomorphism 
\begin{equation}\label{basechange}\F^Y_{M_C,\cR}\cong p_*(\I_{\widetilde\Delta}\otimes q^*(\cR\otimes \nu^*M_C))
\end{equation}
Admitting the claim, the differentiation map (\ref{diff}) is defined as usual, as $p_*$ of the restriction to $\widetilde\Delta$. The isomorphism (\ref{basechange}):
 in the first place $p_*(\I_{\widetilde\Delta}\otimes q^*(\cR\otimes \nu^*M_C))$ is the kernel of the map ($p_*$ of the restriction map)
 \begin{equation}\label{r} p_*q^*(\cR\otimes \nu^*M_{C})\rightarrow p_*q^*((\cR\otimes \nu^*M_{C})_{|\widetilde\Delta})\cong \cR \otimes \nu^*M_{C}
 \end{equation}
 (it is easily seen that the sequence $0\rightarrow \I_{\widetilde\Delta}\rightarrow \OO_{{Y}\times_T{Y}}\rightarrow \OO_{\widetilde\Delta}\rightarrow 0$ remains exact when restricted to $(Y\times T)\times_T (C\times T)$). To prove (\ref{basechange}) we note that, by flat base change, $$\pi_{Y}^*{\pi}_*(\cR\otimes \nu^*M_{C})\cong p_*q^*(\cR\otimes \nu^*M_{C})$$
  and, via such isomorphism, the map  (\ref{r}) is identified to the relative evaluation map.
  
\medskip \noindent\textbf{The isomorphism (\ref{koszul}).  }
  This follows from the spectral sequence
\[\mathrm{Ext}^i_C(\Omega^1_{Y}\otimes N^\on_{C},\, \mathcal{E}xt^j_Y(\OO_C,\OO_Y))\Rightarrow \mathrm{Ext}^{i+j}_Y(\Omega^1_{Y}\otimes N^\on_{C},\OO_Y)\] using that, being $C$ the complete intersection of $n-1$ divisors in  the linear system $|L_{|Y}|$,  we have that $\mathcal{E}xt^j_Y(\OO_C,\OO_Y)=N^{\on-1}_{C}$ if $j=n-1$ and zero otherwise. 
Seeing the elements of $\mathrm{ Ext}$-groups as higher extension classes with their natural multiplicative structure (Yoneda $\mathrm{Ext}$'s, see e.g. \cite{eisenbud} Chapter III), we denote 
\begin{equation}\label{koszulext}\kappa\in \mathrm{Ext}^{n-1}_Y(\OO_C, L^{\otimes -(n-1)}_{|Y})
\end{equation} the extension class of the Koszul resolution of $\OO_C$ as  $\OO_Y$-module
\begin{equation}\label{koszulrel}0\rightarrow L_{|Y}^{\otimes-(n-1)}\rightarrow\cdots\rightarrow {(L_{|Y}^{\otimes-1})}^{\oplus n-1}\rightarrow \OO_Y\rightarrow\OO_C\rightarrow 0 \>\>.
\end{equation}
Then the multiplication with $\kappa$
\[\Ext^1_C(\Omega^1_Y\otimes N_C,\OO_C)\buildrel{\cdot \kappa}\over\rightarrow \Ext^n_Y(\Omega^1_Y\otimes N_C, L_{|Y}^{\otimes -(n-1)})\cong\Ext^n_Y(\Omega^1_Y\otimes N^\on_C,\OO_Y)
\]
is an isomorphism coinciding, up to scalar, with (\ref{koszul})).


\subsection{First step (statement)} 
\begin{notation}\label{abbreviation} From this point we will adopt the hypotheses and the notation  of Theorem \ref{gaussian}. We also adopt the following typographical abbreviations:
$$\F^Y=\F^Y_{N_C^\on\, ,\Q_{|Y\times T}}\qquad\qquad\F^X=\F^X_{N_C^\on\, ,\Q}\qquad\qquad g=g_{N_C^\on, \Q_{|Y\times T}}$$
\end{notation}
The first, and most important, step of the proof of Theorem \ref{gaussian} and Proposition \ref{Aprimo} consists in an explicit calculation of the class $g(e)$. This is the content of Lemma \ref{key1} below. The strategy is as follows. 
Applying $\Ext_{Y\times Y}^n(\>\cdot\>,\Q_{|Y\times T})$ to the basic sequence
\[0\rightarrow \F^Y\rightarrow \pi_{Y}^*\pi_*(\Q\otimes\nu^*N_C^\on)\rightarrow \Q\otimes\nu^*N_C^\on \rightarrow 0\]
(namely (\ref{FY}) for $A=N_C^\on$ and $\cR=\Q_{|Y\times T}$ \footnote{the surjectivity on the right follows from the hypotheses of Theorem \ref{gaussian}}), we get the following diagram with exact (in the middle) column
\begin{equation}\label{A}
\entrymodifiers={+!!<0pt,\fontdimen22\textfont2>}\xymatrixcolsep{1pc}\xymatrixrowsep{1.5pc}\xymatrix{&\Ext^n_{Y\times T}(\Q\otimes\nu^*N_C^\on,\Q_{|Y\times T})\ar[d]^h\\& \Ext^n_{Y\times T}(\pi_{Y}^*{\pi}_*(\Q\otimes\nu^*N_C^\on),\Q_{|Y\times T})\ar[d]^f\\  \Ext^n_{Y}(\Omega^1_Y\otimes N_C^{\otimes n},\OO_{Y})\ar[r]^g&\Ext^n_{Y\times T}(\F^Y,\Q_{|Y\times T})}
\end{equation}
In Definition \ref{defb} below will produce a certain class $b$ in the source of $f$, namely
\begin{equation}\label{b}
b\in \Ext^n_{Y\times T}(\pi_{Y}^*{\pi}_*(\Q\otimes\nu^*N_C^\on),\Q_{|Y\times T}) 
\end{equation}
such that its coboundary map 
$$\delta_b: {\pi}_*(\Q\otimes\nu^*\N_C^\on)\rightarrow R^n\pi_*(\Q_{|Y\times T})$$ is the composition 
\begin{equation}\label{composition}
\xymatrixcolsep{1pc}\xymatrixrowsep{1pc}\xymatrix{{\pi}_*(\Q\otimes\nu^*\N_C^\on)\ar[r]^-\alpha\ar[rd]^-{\delta_b}&R^n\pi_*(\Q)\ar[d]^-\beta\\
&R^n\pi_*(\Q_{|Y\times T})} 
\end{equation}
where the horizontal map $\alpha$ is the coboundary map of the natural extension of $\OO_{X\times T}$-modules
\begin{equation}\label{koszulX}
 0\rightarrow \Q\rightarrow \cdots\rightarrow (\Q\otimes \nu^*L^{\otimes n-1})^{\oplus n}\rightarrow \Q\otimes \nu^*L^\on
 \rightarrow \Q\otimes \nu^*N_C^\on\rightarrow 0
 \end{equation}
($\nu^*$ of the Koszul resolution of $\OO_C$ as $\OO_X$-module, twisted by $\Q\otimes\nu^*L^\on$) and the vertical map $\beta$ is simply $R^n\pi_*$ of the restriction map \ $\Q \rightarrow \Q_{|Y\times T}$ .
 The main Lemma is
 \begin{lemma}\label{key1} 
\  $ f(b)=g(e)$. 
 \end{lemma}
 
 Note that this will already prove Proposition \ref{Aprimo}. Indeed, if $T=\{point\}$ then the vector space of (\ref{b}) is 
\begin{equation}\label{dimAprimo}\Ext^n_{Y}(H^0(\Q\otimes\nu^*N_C^\on)\otimes \OO_Y,\Q_{|Y})\cong \Hom_k(H^0(\Q\otimes\nu^*N_C^\on),
H^n(Y,\Q_{|Y}))
\end{equation}
hence the class $b$ coincides, up to scalar, with its coboundary map $\delta_b$. 
From the description of $\delta_b$ we have  that  $\delta_b=0$ if $H^n(X,\Q)=0$. If this is the case, then Lemma \ref{key1} says that $g(e)=0$, proving  Proposition \ref{Aprimo} in this case. If $\dim T>0$  we consider the spectral sequence 
\[\Ext^i_{ T}({\pi}_*(\Q\otimes N_C^\on) , R^{j}{\pi}_*(\Q_{|Y\times T}))
\Rightarrow\Ext^{i+j}_{Y\times T}
(\pi_{Y}^*{\pi}_*(p^*(\Q\otimes \nu^*N_C^\on)) ,\Q_{|Y\times T})\]
coming from the isomorphism 
\[\R\Hom_{T}({\pi}_*(\Q\otimes N_C^\on) , \R{\pi}_*(\Q_{|Y\times T})) \cong
 \R\Hom_{Y\times T}
(\pi_{Y}^*{\pi}_*(p^*(\Q\otimes \nu^*N_C^\on)) ,\Q_{|Y\times T})\]
Since we are assuming that $R^i\pi_*(\Q_{|Y\times T})=0$ for $i<n$, the spectral sequence degenerates providing an isomorphism as (\ref{dimAprimo}), and Proposition \ref{Aprimo} follows in the same way.\qed

Next, we give the definition of the class $b$ of (\ref{b}). In order to do so, we introduce some additional notation
\begin{notation}\label{K}We denote ${\mathbf K}_{C,X}^\bullet$ (resp. ${\mathbf H}_{C,Y}^\bullet$) the $\nu^*$ of the Koszul resolution of  the ideal of $C$ is $X$  tensored with $\Q\otimes\nu^*L^\on$ (resp. $\nu^*$ of the Koszul resolution resolution of $\OO_C$ as $\OO_Y$-module, tensored with $\Q\otimes \nu^*L_{|Y}^{\otimes n-1}$ ):
\[\begin{matrix}{\mathbf K}_{C,X}^\bullet&  0\rightarrow& \Q& \rightarrow \cdots \rightarrow(\Q\otimes\nu^*L^{\otimes n-2})^{\oplus {n\choose 2}} \rightarrow&(\Q\otimes \nu^*L^{\otimes n -1})^{\oplus n}\\
{\mathbf H}_{C,Y}^\bullet& 0\rightarrow &\Q_{|Y\times T}&\rightarrow \cdots \rightarrow (\Q\otimes \nu^*L_{|Y}^{\otimes n-2})^{\oplus n-1}\rightarrow &\Q\otimes \nu^*L_{|Y}^{\otimes n-1}\end{matrix}
\]
(note that they have the same length). 
 For example, with this notation  the exact complex of $\OO_{X\times T}$-modules (\ref{koszulX}) is written as
 \begin{equation}\label{KOSZULX}
  {\mathbf K}_{C,X}^\bullet \rightarrow Q\otimes L^\on\rightarrow \Q\otimes \nu^*N_C^\on\rightarrow 0
 \end{equation}
 \end{notation}
 \begin{definition}[The class $b$ of (\ref{b})]\label{defb}
 Composing (\ref{KOSZULX}) with the relative evaluation map of $\Q\otimes \nu^*N_C^\on$ (seen as a sheaf on $X\times T$)
$$\pi^*\pi_*(\Q\otimes \nu^*N_C^\on)\rightarrow \Q\otimes \nu^*N_C^\on$$
we get the commutative exact diagram
\begin{equation}\label{sequence}
\begin{matrix}
  {\mathbf K}_{C,X}^\bullet& \rightarrow&\E&\rightarrow & \pi^*\pi_*(\Q\otimes \nu^*N_C^\on)&\rightarrow 0\\
  \Vert&&\downarrow&&\downarrow\\
  {\mathbf K}_{C,X}^\bullet & \rightarrow & Q\otimes L^\on&\rightarrow& \Q\otimes \nu^*N_C^\on&\rightarrow 0
  \end{matrix}
\end{equation}
where $\E$ is a $\OO_{X\times T}$-module. 
Since $tor_{X\times T}^i(\pi^*\pi_*(\Q\otimes \nu^*N_C^\on),\nu^*\OO_Y)=0$ for $i>0$, restricting the top row of (\ref{sequence}) to $Y\times T$ we get an \emph{exact} complex of $\OO_{Y\times T}$-modules
\begin{equation}\label{sequenceY}
 ( {\mathbf K}_{C,X}^\bullet)_{|Y\times T}\rightarrow\E_{|Y\times T} \rightarrow 
\pi_{Y}^*\pi_*(\Q\otimes \nu^*N_C^\on)\rightarrow 0
\end{equation}
We define the class $b\in \Ext^n_{Y\times T}(\pi_{Y}^*{\pi}_*(\Q\otimes\nu^*N_C^\on),\Q_{|Y\times T}) $ of (\ref{b}) as the extension class of the  exact complex (\ref{sequenceY}). The assertion about its coboundary map follows from its definition. 
\end{definition}
We will need the following
\begin{lemma}\label{claim2}    The row of the following diagram
\begin{equation}\label{sequenceY2}
 \xymatrixcolsep{1pc}\xymatrixrowsep{0.5pc} \xymatrix{{\mathbf H}_{C,Y}^\bullet\ar[rr]\ar[rd]
  &&\E_{|Y\times T} \ar[r]&
\pi_{Y}^*\pi_*(\Q\otimes\nu^*N_C^\on)\rightarrow 0\\
&\Q\otimes \nu^*N_C^{\otimes n-1}\ar[ur]\ar[dr]\\
0\ar[ur]&&0\\} 
\end{equation}
is an exact complex having the same extension class of (\ref{sequenceY}), namely  $b\in \Ext^n_{Y\times T}(\pi_{Y}^*{\pi}_*(\Q\otimes\nu^*N_C^\on),\Q_{|Y\times T}) $.
\end{lemma}
\begin{proof} 
 For $n=1$, i.e. $C=Y$, there is nothing to prove. For $n>1$, recall that, by its definition, the top row of (\ref{sequence}) is 
 \[\xymatrixcolsep{1pc}\xymatrixrowsep{0.5pc}\xymatrix{{\mathbf K}^\bullet_{C,X}\ar[dr]\ar[rr]&&\E\ar[r]&  \pi^*\pi_*(\Q\otimes \nu^*N_C^\on)\rightarrow 0\\ &\I_{C/X}\otimes\Q\otimes \nu^*L^\on \ar[ur]\ar[dr]\\
0\ar[ur]&&0} \]
 Recalling that the curve
   $C$ is the complete intersection $Y_1\cap\dots \cap Y_n$, with $Y_i\in |L|$ , and that $Y=Y_1$, restricting the ideal sheaf $\I_{C/X}$ to $Y$ one gets   $\I_{C/Y}\oplus N_C^{-1}$.
   Accordingly the Koszul resolution of $\I_{C/X}$, restricted to $Y$, splits as  the direct sum sum of the Koszul resolution of $\I_{C/Y}$ and the Koszul resolution of $\OO_C$, as $\OO_Y$-module, tensored with $L_{|Y}^{-1}$:
   \[0\rightarrow \begin{matrix}0\\ \oplus\\ L_{|Y}^{- n}\end{matrix}\rightarrow \cdots\rightarrow \begin{matrix}(L_{|Y}^{-2})^{\oplus {{n-1} \choose 2}}\\
   \oplus\\ (L_{|Y}^{-2})^{\oplus n-1}\end{matrix} \>\>\rightarrow\>\> \begin{matrix}(L_{|Y}^{-1})^{\oplus n-1}\\
   \oplus\\ L_{|Y}^{-1}\end{matrix} \>\>\begin{pmatrix}&\I_{C/Y}\\ \rightarrow&\oplus&\rightarrow 0\\
   &N_C^{-1}\end{pmatrix}\]
   Now restricting the exact complex (\ref{sequence}) to $Y$ one gets the exact complex (\ref{sequenceY}) whose "tail", namely the exact complex  $({\mathbf K}^\bullet_{C,X})_{|Y\times T}$
splits as above. Therefore deleting the exact complex corresponding to the above upper row
 one gets the equivalent -- as extension~-- exact complex (\ref{sequenceY2}). This proves the claim.
\end{proof}

\subsection{First step (proof) } In this subsection  we prove Lemma \ref{key1}. We first compute $g(e)$ \footnote{This argument follows \cite{v} p. 252)}. The exact sequences defining $\F^X$ and $\F^Y$ (see Notation \ref{abbreviation}) fit into the commutative diagram
 $$
 \begin{matrix}0\rightarrow&\F^X&\rightarrow& \pi^*\pi_*(\Q\otimes\nu^*N_C^\on)&\rightarrow& \Q\otimes\nu^*N_C^\on&\rightarrow 0\\
 &\downarrow&&\downarrow&&\Vert\\
 0\rightarrow&\F^Y&\rightarrow& \pi_{Y}^*\pi_*(\Q\otimes\nu^*N_C^\on)&\rightarrow& \Q\otimes\nu^*N_C^\on &\rightarrow 0\end{matrix}
 $$
yielding, after restricting the top row to $Y\times T$, the exact sequence
 \begin{equation}\label{c1}0\rightarrow \Q\otimes \nu^*N_C^{\on-1}\rightarrow ({\F^X})_{|Y\times T}\rightarrow \F^Y\rightarrow 0
 \end{equation}
 where the sheaf on the left is  \ \ $tor_1^{\OO_{X\times T}}(\Q\otimes \nu^*N_C^\on,\OO_{Y\times T})$.

This sequence in turn fits into the
  commutative diagram with exact rows
   $$
 \begin{matrix}0\rightarrow&\Q\otimes \nu^*N_C^{\on-1}\rightarrow\!\!\!\!&({\F^X})_{|Y\times T}&\!\!\!\!\rightarrow\!\!\!\!&\F^Y&\!\!\!\!\!\!\rightarrow 0\\
 &\Vert&\downarrow&&\downarrow\\
 0\rightarrow&\Q\otimes\nu^*N_C^{\on-1}\rightarrow\!\!\!\!&\Q\otimes\nu^*(\Omega^1_X\otimes N_C^\on)&\!\!\!\!\rightarrow\!\!\!\!&\Q\otimes\nu^*(\Omega^1_Y\otimes N_C^\on)&\!\!\!\!\!\rightarrow0
 \end{matrix}
 $$
where the class of the bottow row is $\nu^*(e)\in \nu^* \Ext^1_C(\Omega^1_Y\otimes N_C,\OO_C).$
 It follows that  $g(e)$ (where now $e$ is seen in $\Ext^n_Y(\Omega^1_Y\otimes N^{\otimes n}_C,\OO_Y) $, see (\ref{koszul}) and Subsection \ref{preliminaries})  ) is the class of the sequence (\ref{c1}) with ${\mathbf H}_{C,Y}^\bullet$ attached on the left
  \begin{equation}\label{image}
 {\mathbf H}_{C,Y}^\bullet \rightarrow ({\F^X})_{|Y\times T}\rightarrow \F^Y\rightarrow 0 \>\>.
  \end{equation}
  
  Next, we compute $f(b)$. The exact complex (\ref{sequence}) is the middle row of the commutative exact diagram
    \begin{equation}\label{f(b)}
  \begin{matrix}&&0&& 0\\&&\downarrow&&\downarrow\\
  &&\F^X&\Relbar&\F^X\\
  &&\downarrow&&\downarrow\\
   {\mathbf K}_{C,X}^\bullet&\rightarrow&\E&\rightarrow& 
\pi^*\pi_*(\Q\otimes\nu^*N_C^\on)&\rightarrow 0\\
\Vert&&\downarrow&&\downarrow\\
{\mathbf K}_{C,X}^\bullet& \rightarrow&\Q\otimes\nu^*L^\on&\rightarrow&\Q\otimes\nu^*N_C^\on&\rightarrow0\\
&&\downarrow&&\downarrow\\
 && 0&&0\end{matrix}
  \end{equation}
  This provides us with the commutative exact diagram
  $$\begin{matrix}&0&&0\\ &\downarrow&&\downarrow\\
   ({\mathbf H}_{C,Y}^\bullet)_{|Y\times T}\rightarrow &(\F^X)_{|Y\times T}&\rightarrow&\F^Y&\rightarrow 0\\&\downarrow&&\downarrow\\
   &\E_{|Y\times T}&\rightarrow&\pi_Y^*\pi_*(\Q\otimes\nu^*N_C^\on)&\rightarrow 0\end{matrix}
  $$
 where the long row  is (\ref{image}), whose class is $g(e)$. 
By Lemma \ref{claim2}, we can complete the above diagram as follows
 \[
  \begin{matrix}{\mathbf H}_{C,Y}^\bullet&\rightarrow&({\F^X})_{|Y\times T}&\rightarrow&\F^Y&\rightarrow 0\\
  \Vert&&\downarrow&&\downarrow\\
{\mathbf H}^\bullet_{C,Y}&\rightarrow& \E_{|Y\times T}&\rightarrow&\pi_{Y}^*\pi_*(\Q\otimes\nu^*N_C^\on)&\rightarrow 0
\end{matrix}
\]
  where the class of the bottow row is $b$. By definition, the class of the top row is $f(b)$, and it is equal to $g(e)$. This proves Lemma \ref{key1}.\qed

\subsection{Conclusion of the Proof of Theorem \ref{gaussian}}   The last step is 

\begin{lemma}\label{lemma2} We keep the notation and setting of Lemma \ref{key1}. Assume that the line bundle $L$ on $X$ is sufficiently positive. If $f(b)=0$ then $b=0$. 
\end{lemma}

Assuming this, Theorem \ref{gaussian} follows: if $g(e)=0$ then, by Lemmas \ref{key1} and \ref{lemma2} it follows that   $b=0$, hence its coboundary map $\delta_b=\beta\circ \alpha$ is zero (see (\ref{composition})). Taking $L$ sufficiently positive,   
Serre vanishing yields that $\alpha$ is surjective and $\beta$ is injective. Therefore the target of  $\delta_b$,  namely $R^n\pi_*\Q$, is zero.\qed

\begin{proof} (of Lemma \ref{lemma2}) \ The proof is a somewhat tedious repeated application of Serre vanishing. Going back to diagram (\ref{A}) we have that if $f(b)=0$ then there is a 
\ \ $c\in \Ext^n_{Y\times T}(\Q\otimes\nu^*N^\on_{|C},\Q_{|Y\times T})$ \ \   such that 
\begin{equation}\label{h(c)}
h(c)=b \>\>.
\end{equation}
Now we consider  the  commutative diagram
\begin{equation}\label{aux} \xymatrixrowsep{1.5pc}\xymatrixcolsep{1pc}\xymatrix{\Ext^n_{Y\times T}(\Q\otimes \nu^*N_C^\on,\Q_{|Y\times T})\ar[r]^-r\ar[d]^h&\Ext^n_{X\times T}(\Q\otimes \nu^*L^\on,\Q_{|Y\times T})\ar[d]^{h^\prime} \\
\Ext^n_{Y\times T}(\pi_{Y}^*\pi_*(\Q\otimes \nu^*N_C^\on),\Q_{|Y\times T})\ar[r]^-s\ar[d]^\mu
& \Ext^n_{X\times T}(\pi^*\pi_*(\Q\otimes\nu^*L^\on),\Q_{|Y\times T})\ar[d]^{\mu^\prime} \\
 \mathrm{Hom}_{}({\pi}_*(\Q\otimes\nu^*N_C^\on),R^n{\pi}_*\Q_{|Y\times T})\ar[r]^t&\mathrm{Hom}_{}(\pi_*(\Q\otimes\nu^*L^\on),R^n{\pi}_*\Q_{|Y\times T})}
\end{equation}
where:\\
(a) $h$ is  as above and $h^\prime$ is the analogous map \ $\Ext^n_X(ev_X,\Q_{|Y\times T})$, where $ev_X$ is the relative evaluation map on $X\times T$: \ $\pi^*\pi_*(\Q\otimes\nu^*L^\on)\rightarrow \Q\otimes\nu^*L^\on\>\>.$\\
(b) $\mu$ is  the map taking an extension to its coboundary map.  Consequently 
the map \ $\mu\circ h$ \  takes  an extension class $e\in \Ext^n_{Y\times T}(\Q\otimes\nu^*N_C^\on)\,,\Q_{|Y\times T})$ 
to its coboundary map 
\[{\pi}_*(\Q\otimes\nu^*N_C^\on)\rightarrow R^n{\pi}_*(\Q_{|Y\times T})\]
 The map $\mu^\prime\circ h^\prime$ operates in the same way;\\
(c)  notice that the target of $r$ is simply $H^n(\nu^*L^{\otimes -n}_{|Y\times T})$, i.e. $\Ext^n_{Y\times T}(\nu^*L^\on_{|Y},\OO_{Y\times T})$. Via this identification the map $r$ is defined as the natural map 
\[\Ext^n_{Y\times T}(\nu^*L^\on_{|C},\OO_{Y\times T})\rightarrow \Ext^n_{Y\times T}(\nu^*L^\on_{|Y},\OO_{Y\times T})\]
(d) $s$ and $t$ are the natural maps.

 We know that the coboundary map of the extension class $b$ factorizes through the natural coboundary map $\alpha:\pi_*(\Q\otimes\nu^*N_C^\on)\rightarrow R^n\pi_*(\Q)$. This  implies  that  $(t\circ\mu) (b)=0$.
  Therefore, by (\ref{h(c)}) and (\ref{aux}), we have that $(\mu^\prime\circ h^\prime\circ r) (c)=0$. The Lemma will follow  from
 the fact  that both  $r$ and $\mu^\prime\circ h^\prime$ are injective: 
 
 \noindent \emph{Injectivity of  $r$: } in the case $n=1$, i.e. $Y=C$, the map $r$ is just the identity (compare (c) above). Assume that $n>1$. Chasing in the Koszul resolution of $\OO_C$ as $\OO_Y$-module  one finds that the injectivity of $r$ holds as soon as $\Ext^{n-i}_{Y\times T}(\nu^*L^{\on-i}_{|Y},\OO_{Y\times T})=0$ for $i=1,\dots ,n-1$. But these are simply $H^{n-i}(Y\times T,L^{\otimes i-n}_{|Y}\boxtimes\OO_T)$ and the result follows easily from K\"unneth decomposition, Serre vanishing and Serre duality.

 \noindent \emph{Injectivity of $\mu^\prime\circ h^\prime$: } We have that $\Ext^n_{X\times T}(\Q\otimes\nu^*L^\on,\Q_{|Y\times T})\cong $\break $\cong H^n(Y\times T,L^{-n}_{|Y}\boxtimes\OO_T)$.  If $L$ is a sufficiently positive, it follows  as above that this is isomorphic to $H^n(Y,L^{-n}_{|Y})\otimes H^0(T,\OO_T)$. Therefore the map
 $\mu^\prime\circ h^\prime$ is identified to the $H^0$ of the following map of $\OO_T$-modules
 \begin{equation}\label{T}H^n(Y,L^{-n}_{|Y})\otimes \OO_T\rightarrow \mathcal{H}om_T(\pi_*(\Q\otimes\nu^*L^\on),R^n{\pi}_*\Q_{|Y\times T})
 \end{equation}
Hence the injectivity
 of $\mu^\prime\circ h^\prime$ holds as soon as (\ref{T}) is injective at a general fiber.  
 For a closed point $t\in T$, let $\Q_t=\Q_{|X\times\{t\}}$. By base change the map (\ref{T}) at  a general fiber $X\times \{t\}$ is
 \begin{equation}\label{comult}
H^n(Y,  L^{\otimes - n}_{|Y})\rightarrow H^0(X,Q_t\otimes L^\on)^\vee\otimes H^n(Y, {\Q_t}_{|Y}).
 \end{equation}
which  is the  the Serre-dual of the multiplication map of global sections
\begin{equation}\label{mult}
H^0(X,\Q_t \otimes L^\on)\otimes H^0(Y,(\omega_{X}\otimes \Q_t^{ -1}\otimes L)_{|Y})\rightarrow H^0(Y, (\omega_{X}\otimes L^{\on +1})_{|Y}).
\end{equation}
At this point a standard argument with Serre vanishing shows  that (\ref{mult}) is surjective as soon as $L$ is sufficiently positive\footnote{in brief, one shows that the desired surjectivity follows from the surjectivity of
$H^0(X,\Q_t \otimes L^\on)\otimes H^0(X,\omega_{X}\otimes \Q_t^{ -1}\otimes L)\rightarrow H^0(X, \omega_{X}\otimes L^{\on +1}).$ This in turn is proved by interpreting such a multiplication map as the $H^0$ of a restriction-to-diagonal map of $\OO_{X\times X}$-modules}. This proves the injectivity of $\mu^\prime\circ h^\prime$ and concludes also the proof  of the Lemma.
\end{proof}

\section{Proof of Corollary \ref{B}}\label{section3}
The deduction
of Corollary \ref{B} from Theorem \ref{gaussian} is a standard argument with Serre vanishing. However, there are some complications due to the weakness of the assumptions on the singularities of the variety $X$.

\noindent\textbf{A gaussian map on the ambient variety $\mathbf{X\times T}$. } The argument makes use of  a (dual) gaussian map defined on the ambient variety $X\times T$ itself. Namely, for a line bundle $A$ on $X$ we define $\MM^X_{A,\Q}$ as the kernel of the relative evaluation map 
\[\pi^*\pi_*(\Q\otimes \nu^*A)\rightarrow \Q\otimes \nu^*A\]
As in (\ref{diff}) and Subsection \ref{preliminaries} there is the isomorphism
\begin{equation}\label{basechangeX}\MM^X_{A,\Q}\cong { p_X}_*(\I_{\widetilde\Delta_X}\otimes q_X^*(\Q\otimes \nu^*A))
\end{equation}
(where $ p_X$,   $q_X$ and $\tilde \Delta_X$ denote the projections and the diagonal of $(X\times Y)\times_T (X\times T)$). There is also the differentiation map $\MM^X_{A,\Q}\rightarrow\Q\otimes  \nu^*(\Omega^1_X\otimes A)$. 

Now, taking as $A=L^{\otimes n}$ and taking $\Ext^{n+1}_{X\times T}(\>\cdot\>, \Q\otimes \nu^* L^\vee)$ we get the  desired dual gaussian map on $X$:
$$g_X:\Ext^{n+1}_X(\Omega^1_X\otimes L^{\otimes n+1},\OO_X)\rightarrow \Ext^{n+1}_{X\times T}(\MM_{L^\on,\Q}\, ,Q\otimes\nu^*L^\vee) \>\>.$$
Note that there are natural maps $\MM_{L^\on,\Q}\rightarrow \F^X\rightarrow \F^Y$ (see Notation \ref{abbreviation}).

\noindent\textbf{First step. } We consider the commutative diagram
\begin{equation}\label{mappegaussiane}\xymatrixcolsep{1.5pc}\xymatrixrowsep{1.5pc}\xymatrix{\Ext^n_Y(\Omega^1_{Y}\otimes N_C^\on,\OO_Y)\ar[r]^-g\ar[d]^-\mu&\Ext^n_{Y\times T}(\F^Y_{N_C^\on,\Q_{|Y\times T}},\Q_{|Y\times T})\ar[d]^-\eta\\
\Ext^{n+1}_X(\Omega^1_{X}\otimes L^{\on+1},\OO_X)\ar[r]^-{g_X}& \mathrm{Ext}^{n+1}_{X\times T}(\MM^X_{L^{\on},\Q},  \Q\otimes\nu^*L^{\vee})\\
}
\end{equation}
where, as in the previous section, $g$ denotes the main character, namely the (dual) gaussian map $g_{N_C^\on,\Q_{Y\times T}}$. The maps $\mu$ and $\eta$ are the natural ones, and the definition is left to the reader \footnote{for example, $\eta$ is defined by  sending $\Ext^n_Y(\Omega^1_{Y}\otimes N_C^\on,\OO_Y)$ to $\Ext^n_X(\Omega^1_{Y}\otimes N_C^\on,\OO_Y)$ and then composing   with the natural map $\Omega^1_X\otimes L^\on\rightarrow \Omega^1_Y\otimes N_C^\on$ on the left, and with the natural extension $0\rightarrow L^\vee\rightarrow \OO_X\rightarrow \OO_Y\rightarrow 0$ on the right}. However such maps are more easily understood by considering the following commutative diagram, whose maps are the natural ones
\begin{equation}\label{duali}
\xymatrixcolsep{1pc}\xymatrixrowsep{1.5pc}\xymatrix{H^d(\MM^X_{L^{\on},\Q}\otimes\! \Q^\vee\!\otimes\! ((\omega_X\otimes L)\boxtimes \omega_T)))
\ar[r]^-{g_X^\prime}\ar[d]^{\eta^\prime}&
H^0(\Omega^1_X\otimes L^{\on+1}\otimes\omega_X)\!\otimes\! H^d(\omega_T)\ar[d]^{\mu^\prime}\\
H^d(\F^Y_{N_C^\on,\Q}\otimes\! \Q^\vee\!\otimes\!(\omega_Y\boxtimes\omega_T))\ar[r]^-{g^\prime}&H^0(\Omega^1_Y\otimes L^\on\otimes \omega_Y)\!\otimes\! H^d(\omega_T)}
\end{equation}
where $d=\dim T$. 
 (Note that,
since $X$ is Cohen-Macaulay, adjunction formulas for dualizing sheaves do hold). Notice also that, if $X$ and $T$ are  Gorenstein, then (\ref{duali}) is the dual of diagram (\ref{mappegaussiane}). 

\smallskip
As it is easy to see, after tensoring with $\omega_C\otimes N_C$ the restricted normal sequence (\ref{restricted normal}) remains exact:
  \begin{equation}
  \label{stillexact}
  0\rightarrow \omega_C\rightarrow \Omega^1_X\otimes L\otimes \omega_C\rightarrow \Omega^1_Y\otimes N_C\otimes \omega_C\rightarrow 0
  \end{equation}
  Therefore $e$ defines naturally a linear functional on $H^0(\Omega^1_Y\otimes N_C\otimes \omega_C)$ (compare also (\ref{nonGoreII}) below), still denoted by $e$. We have
 \begin{claim}\label{claim}
  If $L$ is sufficiently positive, then the map $g_X^\prime$ is surjective, while $coker \mu^\prime$ is one-dimensional, with $(coker \mu^\prime)^\vee$  spanned by $e$.
  \end{claim}
  \begin{proof}  Serre vanishing ensures the surjectivity of the restriction 
  \[H^0(\Omega^1_X\otimes L^{\on+1}\otimes\omega_X)\rightarrow H^0( \Omega^1_X\otimes L\otimes \omega_C)\]
  Since the map $\mu^\prime$ is the composition of the above map with $H^0$ of the right arrow of sequence (\ref{stillexact}), the Claim for $\mu^\prime$  follows.

  Concerning the surjectivity of the map $g_X^\prime$, we first note that
by Serre vanishing,
\begin{equation}\label{locallyfree}
R^ip_*(\I_{\widetilde\Delta_{X}}\otimes q_X^*(\nu^*L^\on\otimes\Q))=\begin{cases}0 & \hbox{for $i>0$}\\
\hbox{locally free}&\hbox{for $i=0$}\end{cases}
\end{equation}
Now we project on $T$. A standard computation using (\ref{locallyfree}), base-change, Serre vanishing, Leray spectral sequence and K\"unneth decomposition shows that the  map $g_X^\prime$ is identified to
\begin{multline}\label{hq}H^d\bigl(T,\pi_*\bigl( p_*\bigl(\I_{\widetilde\Delta_X}\otimes q^*(\nu^*L^\on\otimes\Q)\bigr)\otimes \Q^\vee\otimes ((L \otimes\omega_X)\boxtimes \omega_T)\bigr)\bigr)\rightarrow\\
\rightarrow H^d( T,H^0(X,\Omega^1_{X}\otimes L^{\on +1}\otimes\omega_X)\otimes\omega_T)
\end{multline}
This the $H^d$ of a map of coherent sheaves on the $d$-dimensional variety $T$. Hence the surjectivity of (\ref{hq})
is implied by the generic surjectivity of the map of sheaves itself. By base change, at a generic fibre $X\times{t}$ such  map of sheaves is the gaussian map
\[\gamma_t: H^0(X, p_*(\I_{\Delta_{X}}\otimes q^*(L^\on\otimes\Q_t))\otimes L\otimes\Q_t^{\vee}\otimes\omega_{X})\rightarrow H^0(X, \Omega^1_{X}\otimes L^{\on +1}\otimes \omega_X)
\]
The map $\gamma_t$ is defined by restriction to the diagonal in the usual way. Once again it follows from relative Serre vanishing (on $(X\times T)\times_T (X\times T)$) that, as soon as $L$ is sufficiently positive,  $\gamma_t$ is surjective for all $t$. This proves the surjectivity of the $g_X^\prime$ and concludes the proof of the Claim.
\end{proof}

\noindent\textbf{Last step. } If $C$ is Gorenstein, Claim \ref{claim} achieves the proof of  Corollary \ref{B}. Indeed, diagram  (\ref{mappegaussiane}) is dual to diagram (\ref{duali}) and it follows that the kernel of our map $g=g^Y_{N_C^\on,\Q}$ is  at most one-dimensional, spanned by $e$. In the general case,  Corollary \ref{B}  follows in the same way once proved the following
\begin{claim}\label{additional} As soon as $L$ is sufficiently positive, the maps $g_X^\prime$ and $\mu^\prime$ are respectively Serre-dual of the maps $g_X$ and $\mu$.
\end{claim} 
\begin{proof} To prove this assertion for $g_X^\prime$ we note that, concerning its source, the sheaf $\MM^X_{L^{\on},\Q}$ is locally free by (\ref{locallyfree}). Therefore
\begin{gather}\label{nuius}\mathrm{Ext}^{n+1}_{X\times T}(\MM^X_{L^{\on},\Q},  \Q\otimes\nu^*L^{\vee})
\cong H^{n+1}(
(\MM^X_{L^{\on},\Q})^{\vee}\otimes \Q\otimes\nu^*L^{\vee})\cong \\
\cong H^{d}(\MM^X_{L^{\on},\Q}\otimes
\Q^\vee\otimes((L\otimes \omega_X)\boxtimes\omega_T))^\vee\end{gather}
Next, we show the Serre duality
\begin{equation}\label{nonGore}H^0(X,\Omega^1_X\otimes\omega_X\otimes L^{\on+1})^\vee \cong\Ext^{n+1}_X({\Omega^1_X}\otimes L^\on,L^{\vee})
\end{equation}
By definition of dualizing complex (see e.g. \cite{rd}, Ch.V,\S2, Prop. 2.1 at p. 258), in the derived category of $X$ we have that $\OO_X=
 \R\mathcal{H}om(\omega_X,\omega_X)$. Therefore it follows that
\begin{gather*}\R\mathrm{Hom}_X(\Omega^1_X\otimes L^{\on+1},\OO_X)=\R\mathrm{Hom}_X(\Omega^1_X\otimes L^{\on+1},\R\mathcal{H}om(\omega_X,\omega_X))=\\
=\R\mathrm{Hom}_X(\Omega^1_X\underline\otimes^{\mathbf{L}}\omega_X\otimes L^{\on+1},\omega_X)
\end{gather*}
By Serre-Grothendieck duality, this is isomorphic to
\[\R\mathrm{Hom}_k(
 \R\Gamma(X,\Omega^1_X\underline\otimes^{\mathbf{L}}\omega_X\otimes L^{\on+1}[n+1]),k)
\]
The spectral sequence computing $ \R\Gamma(X,(\Omega^1_X\underline\otimes^{\mathbf{L}}\omega_X\otimes L^{\on+1})$ degenerates  
to the isomorphisms
\[H^i(X,\Omega^1_X\underline\otimes^{\mathbf{L}}\omega_X\otimes L^{\on+1})\cong \bigoplus_i H^0(X,tor_i^X(\Omega^1_X,\omega_X)\otimes L^{\on+1})\]
(if $L$ is sufficiently positive, by Serre vanishing there are only $H^0$'s). 
 Therefore (\ref{nonGore}) follows. By (\ref{nuius}) and (\ref{nonGore}) we have proved the part of  the Claim concerning $g_X^\prime$. 
 
 Concerning $\mu^\prime$, at this point it is enough to prove the Serre duality
 \begin{equation}\label{nonGoreII}
 \Ext^n_Y({\Omega^1_X}_{|C}\otimes L^n,\OO_Y)\cong \Ext^1_C({\Omega^1_Y}_{|C}\otimes L,\OO_C)\cong H^0(\Omega^1_C\otimes N\otimes\omega_C)^\vee
 \end{equation}
 where the first isomorphism is (\ref{koszul}). Arguing as above, it is enough to prove that the $\OO_C$-modules
 \[tor^i_C((\Omega^1_Y)_{|C},\omega_C)\otimes N_C\]
 have vanishing higher cohomology for all $i$. For $i>0$ this follows simply because they are supported on points. For $i=0$ note that, by the exact sequence (\ref{stillexact}), it is enough to show that 
 \begin{equation}\label{required}
 H^1(\Omega_X^1\otimes N_C\otimes\omega_C)=0
 \end{equation} 
 To prove this, we tensor the Koszul resolution of $\OO_C$ as $\OO_X$-module with $\Omega^1_X\otimes \omega_X\otimes L^{\on +1}$, getting 
 a complex (exact at the last step on the right)
 \[0\rightarrow \Omega^1_X\otimes\omega_X\otimes L\rightarrow \cdots \rightarrow (\Omega^1_X\otimes\omega_X\otimes L^\on)^{\oplus n}\rightarrow\Omega^1_X\otimes\omega_X\otimes L^{\on +1}\rightarrow \Omega^1_X\otimes \omega_C\otimes N_C\rightarrow 0
 \]
 Since $C$ is not contained in the singular locus of $X$, the homology sheaves are supported on points. Therefore 
 the required vanishing (\ref{required}) follows from Serre vanishing via a diagram-chase. This concludes the proof of Claim \ref{additional} and of Corollary \ref{B}.
 \end{proof}

\section{Gaussian maps and the Fourier-Mukai transform}\label{section4}
In this section we will  describe the setup of the proof of Theorem \ref{subvarieties}. We will show that when the variety $X$ is a subvariety of an abelian variety $A$, the parameter variety $T$ is the dual abelian variety $\widehat A$, and the line bundle $\Q$ is the restriction to $X\times \widehat A$ of the Poincar\`e line bundle then  the (dual) gaussian map $g_{N_C^\on ,\Q_{|Y\times T}}$ of the Introduction can be naturally intepreted as a piece of a  (relative version of) the classical Fourier-Mukai transform associated to the Poincar\`e line bundle, applied to a certain space of morphisms.

\begin{notation/assumptions}\label{notationFM} We will keep all the notation and hypotheses of the Introduction. Explicitly: 

\noindent - let $X$ be a $n+1$-dimensional normal Cohen-Macaulay subvariety of a $d$-dimensional abelian variety $A$. As usual we choose an ample line bundle $L$ on $X$ such that we can find $n$ irreducible divisors $Y=Y_1,\dots ,Y_n\in X$ such that their intersection is an irreducible curve $C$. We assume also that $C$ is not contained in the singular locus of $Y$.  The line bundle $L_{|C}$ is denoted $N_C$.\\
- Let $\cP$ be a  Poincar\'e line bundle on $A\times \widehat A$. We denote
$$\Q=\cP_{|X\times \widehat A}\qquad\hbox{and}\qquad \cR=\cP_{|Y\times \widehat A}$$ 
- $\nu$ and $\pi$ are the projections of $Y\times \widehat A$.\\
- We assume that the line bundle $\nu^*L^\on\otimes\Q$ is relatively base point-free, namely the evaluation map $\pi^*\pi_*(\nu^*L^\on\otimes \Q)\rightarrow \nu^*L^\on\otimes\Q$ is surjective (here $\nu$ and $\pi$ denote also the projection of $X\times \widehat A\rightarrow \widehat A$).\\
- $p$, $q$ and $\widetilde\Delta$ are the projections and the diagonal of $(Y\times \widehat A)\times_{\widehat A}(Y\times \widehat A)$.\\
- The gaussian map  of the Introduction (see (\ref{specialized2}))  is
\begin{equation}\label{newgaussian}g=g_{N_C^\on,\cR}: \Ext^n_Y(\Omega^1_Y\otimes N_C^\on,\OO_Y)
\rightarrow \Ext^n_{Y\times \widehat A}(p_*(q^*(\I_{\widetilde\Delta}\otimes\cR\otimes\nu^*N_C^\on)),\cR)
\end{equation}
obtained as (the restriction to the relevant K\"unneth direct summand) of \break $\Ext^n_{Y\times\widehat A}(\>\cdot\>,\cR)$ of the differentiation (i.e. restriction to the diagonal) map (see also Subsection \ref{preliminaries}). We recall also the identification of the source:
$$\Ext^n_Y(\Omega^1_Y\otimes N_C^\on,\OO_Y)\cong\Ext^1_C(\Omega^1_Y\otimes N_C,\OO_C)$$
(see (\ref{koszul}) and Subsection \ref{preliminaries}).\\
- The projections of $Y\times A$ will be denoted $p_1$ and $p_2$.
\end{notation/assumptions}

\begin{remark} Since the variety $X$ is assumed to be smooth in codimension 1, and our arguments will concern a sufficiently positive line bundle $L$, we could have assumed from the beginning that the curve $C$ is smooth and the divisor $Y$ is smooth along $C$.
 However we preferred to assume the smoothness of $C$ only where needed, namely at the end of the proof. See also Remarks \ref{sernesidue} and  \ref{sernesi} below.
\end{remark}

\noindent\textbf{Fourier-Mukai transform. }
 Now we consider the trivial abelian scheme  $Y\times A\rightarrow Y$ and its dual $Y\times \widehat A\rightarrow Y$. The Poincar\'e line bundle $\cP$ induces naturally a Poincar\'e line bundle $\TP$ on $(Y\times A)\times_Y (Y\times\widehat A)$ (namely the pullback of $\cP$ to $Y\times A\times \widehat A$) and we consider the functors
$$\R\Phi:\DD(Y\times A)\rightarrow \DD(Y\times  \widehat A)\qquad\hbox{and}\qquad\R\Psi:\DD(Y\times \widehat A)\rightarrow \DD(Y\times A)$$
defined respectively by ${\R \pi_{Y\times \widehat A}}_*(\pi_{Y\times A}^*(\>\cdot\>)\otimes \tilde P)$ and ${\R \pi_{Y\times A}}_*(\pi_{Y\times \widehat A}^*(\>\cdot\>)\otimes \tilde P)$. By Mukai's theorem (\cite{mu2} Thm 1.1) they are equivalences of categories, more precisely
\begin{equation}\label{inversion}\R\Psi\circ\R\Phi\cong (-1)^*[-q]\qquad\hbox{and}\qquad \R\Phi\circ\R\Psi\cong (-1)^*[-q]
\end{equation}
 In particular it follows that, given $\OO_{Y\times A}$-modules $\F$ and $\G$,  we have the functorial isomorphism
\begin{equation}\label{perseval}FM_i:\Ext^i_{Y\times A}(\F,\G)\buildrel\cong\over\longrightarrow \Ext^i_{Y\times\widehat A}(\R\Phi (\F),\R\Phi (\G))
\end{equation}
(note that the $\Ext$-spaces on the right are usually hyperexts).

\noindent\textbf{The gaussian map. } Now we focus on the target of the gaussian map (\ref{newgaussian}). Let $\Delta_Y\subset Y\times A$ be the graph of the embedding $Y\hookrightarrow A$. In other words, $\Delta_Y$ is the diagonal of $Y\times Y$, seen as subscheme of $Y\times A$. It follows from the definitions that
\begin{equation}\label{k(0)}
\R\Phi(\OO_{\Delta_Y})=\cP_{|Y\times\widehat A}=\cR
\end{equation}
Moreover, we have that
\begin{equation}\label{R(0)}p_*(q^*(\I_{\widetilde\Delta}\otimes\cR\otimes\nu^*N_C^\on))\cong R^0\Phi(\I_{\Delta_Y}\otimes p_2^*N_C^\on)
\end{equation}
This is because of the natural isomorphisms
\[(Y\times Y)\times_Y(Y\times \widehat A)\cong Y\times Y\times \widehat A\cong (Y\times \widehat A)\times_{\widehat A}(Y\times \widehat A)\]
 yielding the identifications $\tilde \cP_{|(Y\times Y)\times_Y (Y\times \widehat A)}\cong q^*(\cP_{|Y\times\widehat A})=q^*(\cR)$. 
Moreover, for any sheaf $\F$ supported on $Y\times C$ (as $\I_{\Delta_Y}\otimes p_2^*N_C^\on$ ), we have that $R^i\Phi(\F)=0$ for $i>1$. Therefore the fourth quadrant spectral sequence 
\[\Ext^p_{Y\times \widehat A}(R^q\Phi(\F),\cR)\Rightarrow\Ext^{p-q}(\R\Phi(\F),\cR)\]
is reduced to a long exact sequence
\begin{multline}\label{spectral1}\cdots \rightarrow \Ext^{i-1}_{Y\times \widehat A}(R^0\Phi(\F),\cR)\rightarrow \Ext^{i+1}_{Y\times \widehat A}(R^1\Phi(\F),\cR)\rightarrow\\ 
\rightarrow \Ext^{i}_{Y\times \widehat A}(\R\Phi(\F),\cR)\rightarrow\Ext^{i}_{Y\times\widehat A}(R^0\Phi(\F),\cR)\rightarrow\cdots
\end{multline}
Putting  all that together we get the following diagram, with right column exact in the middle
\begin{equation}\label{subvarietiesFM}\xymatrixcolsep{1.5pc}\xymatrixrowsep{1.5pc}\xymatrix{\Ext^n_Y(\Omega^1_Y\otimes N_C^\on,\OO_Y)\ar[d]^=\\
\Ext^n_{\Delta_Y}((\I_{\Delta_Y}\!\!\otimes p_2^*N_C^\on)_{|\Delta_Y},\OO_\DDelta)\ar[d]^u&
\Ext^{n+1}_{Y\times\widehat A} (
 R^1\Phi_{\TP}(\I_{\Delta_Y}\!\!\otimes p_2^*N_C^\on),\!\cR)\ar[d]^\alpha\\
 \mathrm{Ext}^n_{Y\times A}(\I_{\Delta_Y}\!\!\otimes p_2^*N_C^\on,\OO_{\Delta_Y})\ar[r]_-\cong^-{FM_n}&
\mathrm{Ext}^n_{Y\times \widehat A}(\R\Phi_{\TP}(\I_{\Delta_Y}\!\!\otimes p_2^*N_C^\on),\!\cR)\ar[d]^\beta\\
  &\Ext^n_{Y\times\widehat A} (
 R^0\Phi_{\TP}(\I_\DDelta\!\!\otimes p_2^*N_C^\on),\!\cR)}
 \end{equation}
 where $u$ is the natural map (see also (\ref{u}) below). In conclusion, the kernel of the gaussian map can be described as follows
 \begin{lemma}\label{firstreduction} The gaussian map $g=g_{N_C^\on,\cR}$ of (\ref{newgaussian}) is the composition $\beta\circ FM_n\circ u$. Therefore 
 \[\ker (g)\cong Im(FM_n\circ u)\cap Im (\alpha)\]
 \end{lemma}
 \begin{proof} The identification of the two maps follows using (\ref{R(0)}), simply because they are defined in the same way.\end{proof}

\section{Cohomological computations on $Y\times A$}\label{section5}

In this section we describe the source of the Fourier-Mukai map $FM_n$ of diagram (\ref{subvarietiesFM}) above, together with other related cohomology groups. We will use   the Grothendieck duality (or change of rings) spectral sequence
 \begin{equation}\label{spec2}
\mathrm{Ext}^i_{Y\times C}(\F,\EExt^j_{Y\times A}(\OO_{Y\times C},\OO_\DDelta ))\Rightarrow 
\mathrm{Ext}^{i+j}_{Y\times A}(\F, \OO_\DDelta)
\end{equation}
With this in mind, we compute the sheaves $\EExt^i_{Y\times A}(\OO_{Y\times C},\OO_\DDelta)$ in Proposition \ref{key} below.

\subsection{Preliminaries} The following standard identifications will be useful
\begin{equation}\label{prep1} \bigoplus_i \EExt^i_{Y\times A} (\OO_{\DDelta},\OO_{\DDelta})\cong \bigoplus_i{\delta}_*(\Lambda^i T_{A,0}\otimes\OO_{Y}) 
\end{equation}
(as graded algebras), where $T_{A,0}$ is the tangent space of $A$ at $0$ and $\delta$ denotes the diagonal  embedding 
$$\delta: Y\hookrightarrow Y\times A\>\>.$$
 This holds because $\DDelta$ is the preimage of $0$ via the difference map $Y\times A\rightarrow A$, $(y,x)\mapsto y-x$ (which is flat), and $\EExt^\bullet_A(k(0),k(0))$ is $\Lambda^\bullet T_{A,0}\otimes k(0)$.

 Moreover, letting   $\Delta_C\subset Y\times C\subset Y\times A$  the diagonal of $C\times C$ (seen as a subscheme of $Y\times A$) we have
\begin{equation}\label{prep2}\bigoplus_i\EExt^i_{Y\times A}(\OO_{\Delta_C},\OO_{\DDelta})\cong\bigoplus_i {\delta}_*(\Lambda^{i-n+1}T_{A,0}\otimes N_{C}^{\otimes n-1})
\end{equation}
(as graded modules on the  above algebra).
This is seen as follows: since $C$ is the complete intersection of $n-1$ divisors  of $Y$, all of them in $|L_{|Y}|$, $\EExt^j_{\DDelta}(\OO_{\Delta_C},\OO_{\DDelta})=0$ if $j\ne n-1$ and $\EExt^{n-1}_{\DDelta}(\OO_{\Delta_C},\OO_{\DDelta})={\delta}_*N_C^{\on-1}$.
Therefore (\ref{prep2}) follows from (\ref{prep1}) and the spectral sequence
$$
 \EExt^h_{\DDelta}(\OO_{\Delta_C},\EExt^j_{Y\times A}(\OO_{\DDelta},\OO_{\DDelta}))\Rightarrow \EExt^{h+j}_{Y\times A}(\OO_{\Delta_C},\OO_{\DDelta}).$$
 
\subsection{The (equisingular) restricted normal sheaf} We consider the $\OO_C$-module $\N^\prime$ defined by the sequence
  \begin{equation}\label{normal}
  0\rightarrow (\T_Y)_{|C}\rightarrow (\T_A)_{|C}\rightarrow \N^\prime\rightarrow 0
  \end{equation}
  When $Y=C$ the sheaf $\N^\prime$ is usually called the \emph{equisingular normal sheaf} (\cite{sernesi} Prop. 1.1.9).  Therefore we will refer to $\N^\prime$ as the restricted equisingular normal sheaf. 
  
  \begin{remark}\label{sernesidue} Note that, since $X$ is non-singular in codimension one then the curve $C$ can be taken to be smooth and the divisor $Y$ smooth along $C$ so that $\N^\prime$ is locally free and it is the restriction to $C$ of the normal sheaf of $Y$. Eventually we will make this assumption in the last section. However the computations of the present section work in the more general setting.
  \end{remark}
  
  The sheaves $\EExt^j_{Y\times A}(\OO_{Y\times C},\OO_\DDelta )$ appearing in (\ref{spec2}) are described as follows
  
  \begin{proposition}\label{key}\emph{(a)}
  $\bigoplus_i\EExt^i_{Y\times A}(\OO_{Y\times C},\OO_{\DDelta})\!\cong\! \bigoplus_i{\delta}_*(\overset{i-n+1}\Lambda \N^\prime\otimes N_C^{\on-1})$
  (as graded modules on the algebra (\ref{prep1})).
 In particular the left hand side  is zero for $i<n-1$.\\
 \emph{(b)} \ \ \ \ $\EExt^{d-1}_{Y\times A}(\OO_{Y\times C},\OO_\DDelta)\cong \delta_*\omega_C\>\>.$
\end{proposition}

\proof 
(a) We apply $\mathcal{H}om_{Y\times A}(\>\cdot\>,\OO_{\Delta_Y})$
 to the basic exact sequence
\begin{equation}\label{trivial1}
 0\rightarrow \I_{\Delta_C/Y\times C}\rightarrow \OO_{Y\times C}\rightarrow \OO_{\Delta_C}\rightarrow 0
\end{equation}
where $ \I_{\Delta_C/Y\times C}$ denotes the ideal of $\Delta_C$ in $Y\times C$.
Since $\Delta_C$ is the intersection (in $Y\times A$)  of $\DDelta$ and $Y\times C$, the resulting long exact sequence is chopped in short exact sequences (where we plug the isomorphism  (\ref{prep2}))
\[\label{sequences}
0\rightarrow\EExt^{i-1}_{Y\times A}(\I_{\Delta_C/Y\times C},\OO_{\DDelta})\rightarrow \delta_*(\overset{i-n+1}\Lambda T_{A,0}\otimes N_C^{\on-1})
\rightarrow \EExt^{i}_{Y\times A}(\OO_{Y\times C},\OO_{\DDelta})\rightarrow 0 
\]
This proves that  
\begin{equation}\label{cases}\EExt^{i}_{Y\times A}(\OO_{Y\times C},\OO_{\Delta_Y})=
\begin{cases} 0&\hbox{if $i<n-1$}\\
{\delta}_*N^{\on-1}_{C}&\hbox{if $i=n-1$}
\end{cases}
\end{equation}
For $i=n$ it follows from (\ref{cases}) and  the spectral sequence (\ref{spec2}) applied to $\I_{\Delta_C/Y\times C}$ that  
\[
\EExt^{n-1}_{Y\times A}(\I_{\Delta_C/Y\times C},\OO_{\Delta_Y})
\cong{\mathcal H}om_{Y\times C}({\I_{\Delta_C/Y\times C},\EExt^{n-1}_{Y\times A}(\OO_{Y\times C},\DDelta}))\cong {\delta}_*({\T}_Y\otimes N^{\on-1}_{C})
\]
 and that  $\delta_*$ identifies (\ref{sequences})   to (\ref{normal}), tensored with  $N^{n-1}_{|C}$, i.e.
\begin{equation}\label{normal2}
0\rightarrow {\T}_Y\otimes N_C^{\on-1}\rightarrow T_{A,0}\otimes N^{\on-1}_{C}\rightarrow \mathcal{N}^\prime\otimes N_C^{\on-1}\rightarrow 0
\end{equation}
This proves the statement for $i=n$. For $i>n$,   Proposition \ref{key} follows by induction. Indeed $\EExt^\bullet_{Y\times A}(\OO_{Y\times C},\OO_{\DDelta})
$
is naturally  a graded-module over the exterior algebra 
$\EExt^\bullet_{Y\times A}(\OO_{\Delta_Y},\OO_{\DDelta})\cong {\delta}_*\bigl(\Lambda^\bullet T_{A,0}\otimes\OO_Y\bigr) $ (see (\ref{prep1})). Assume  that the statement of the present Proposition holds for the positive integer  $i-1$. Because of the action of the exterior algebra,  sequences (\ref{sequences}) and (\ref{normal2}) yield that the kernel of the map
\[
\delta_*(\overset{i-n+1}\Lambda T_{A,0}\otimes N^{\on-1}_{C})\rightarrow
 \EExt^{i}_{Y\times A}(\OO_{Y\times C},\OO_{\Delta})\rightarrow 0
\]
is surjected (up to twisting with $N_C^{\on-1}$) by $\delta_*(\Lambda^{i-n}T_{A,0}\otimes ({\T}_Y)_{|C})$. This presentation yields that
$ \EExt^{i}_{Y\times A}(\OO_{Y\times C},\OO_{\Delta})$ is equal to $\delta_*\bigl((\overset{i-n+1}\Lambda {\mathcal N}^\prime)\otimes N_C^{\on-1}\bigr)$. This proves (a).

\noindent (b) If $Y$ is smooth along $C$ then $\N^\prime$ is locally free (coinciding with the restricted normal bundle) (see Remark \ref{sernesidue}). In this case (b) follows at once from (a). In the general case the proof is as follows. We claim that for each $i$ the left hand side of (a) can be alternatively described as 
$$\EExt^i_{Y\times A}(\OO_{Y\times C},\OO_{\DDelta})\cong {\mathcal T}or^{Y\times A}_{d-1-i}(p_2^*\omega_C,\OO_{\DDelta})$$
This is proved by means of the isomorphism of functors
$$\R\Hom_{Y\times A}(\OO_{Y\times C},\OO_\DDelta)\cong \R\Hom_{Y\times A}(\OO_{Y\times C},\OO_{Y\times A})\underline\otimes_{Y\times A}^{\bf L}\OO_\DDelta$$
and the corresponding spectral sequences.
In fact, since $C$ is Cohen-Macaulay, we have that $\EExt^i_{Y\times A}(\OO_{Y\times C},\OO_{Y\times A})=0$ for $i\ne d -1$
and equal to $p_2^*\omega_C$ for $i=d-1$. Thus the spectral sequence computing the right hand side degenerates, proving the claim. 
In particular, for $i=d-1$, we have 
 $\EExt^{d-1}_{Y\times A}(\OO_{Y\times C},\OO_\DDelta)\cong (p_2^*\omega_C)
 \otimes\OO_\DDelta\cong{\delta}_*\omega_C\>.$\endproof

\subsection{Reduction of the statement of Theorem \ref{subvarieties}} As a first application of Proposition \ref{key}, we reduce the statement of Theorem \ref{subvarieties} -- in the equivalent formulation provided by Lemma \ref{firstreduction} -- to a simpler one. This will involve the issue of  comparing two spaces of first-order deformations mentioned in the Introduction (Subsection \ref{subsection3}), and it will be the content of Proposition \ref{secondreduction} and Corollary \ref{corollario} below. 

\begin{notation}\label{notazione mappe} We consider the first spectral sequence (\ref{spec2}) applied to $\F=p_2^*N^\on_{C}$, rather than to $\I_{\Delta}\otimes p_2^*N^\on_{C}$. Plugging the identification provided by Lemma \ref{key} we get
\[H^j(C, \Lambda^{i-n+1} \N^\prime\otimes N_C^{-1})\Rightarrow
 \Ext^{j+i}_{Y\times A}(p_2^*N^\on_{C},\OO_{\DDelta})\]
Since the $H^i$'s on the left are zero for $i\ne 0,1$, the spectral sequence is reduced to short exact sequences
\begin{equation}\label{IIspecsecond}0 \rightarrow H^1(C, \Lambda^{i-n} \N^\prime\otimes N_C^{-1})\xrightarrow{v_i}  \Ext^{i}_{Y\times A}(p_2^*N^\on_{C},\OO_{\DDelta})\xrightarrow{w_i}
 H^0(C, \Lambda^{i-n+1} \N^\prime\otimes N_C^{-1})\rightarrow 0\end{equation}
In particular, for $i=n$, we have  the exact sequence
\[0\rightarrow H^1(C,  N^{-1}_{C})\buildrel {v_n}\over\rightarrow  \Ext^{n}_{Y\times A}(p_2^*N^\on_{C},\OO_{\DDelta})\buildrel{w_n}\over\rightarrow H^0(C,  \N^\prime\otimes N_C^{-1})\rightarrow 0\]
Combining with the exact sequence coming from the spectral sequence (\ref{spectral1}), applied to $\F=p_2^*N_C^{\otimes n-1}$ we get       
\begin{equation}\label{diagramn}
\xymatrixrowsep{1.5pc}\xymatrixcolsep{2pc}\xymatrix{
H^1(C,  N^{-1}_{C})\ar@{>->}[d]^{v_n} &
 \Ext_{Y\times \widehat A}^{n+1}(R^1\Phi(p_2^*N^\on_{C}),\cR)\ar[d]^{a_n}\\
\Ext^{n}_{Y\times A}(p_2^*N^\on_{C},\OO_{\Delta_Y})\ar[r]_{\cong}^-{FM_n}\ar@{->>}[d]^{w_n}&\Ext^{n}_{Y\times \widehat A} (\R\Phi(p_2^*N^\on_{C}),\cR)\ar[d]^{b_n}\\
H^0(C,\N^\prime\otimes N_C^{-1})&\Ext_{Y\times \widehat A}^{n}(R^0\Phi(p_2^*N^\on_{C}),\cR)\\
}
\end{equation}
\end{notation}
\begin{remark}\label{mappaG} Note that, as shown by the exact sequence (\ref{normal}) defining the restricted equisingular normal sheaf,  we get that
\[H^0(C,\N^\prime \otimes N_C^\vee)=\ker (H^1(C,\T_Y\otimes N_C^\vee)\buildrel G\over\rightarrow H^1(C,\T_A\otimes N_C^\vee))\]
This map $G$ is the (restriction to $H^1(C,\T_Y\otimes N_C^\vee)$) of the map $G^Y_A$ of (\ref{general}) in the Introduction (see also Remark \ref{e} below).
\end{remark}

\begin{proposition}\label{secondreduction} 
In diagram (\ref{diagramn}), if the map \ $a_n$ is non-zero then the map $w_n\circ FM_n^{-1}\circ a_n$ is non-zero and its image is contained in the kernel of the gaussian map (\ref{newgaussian}). 
\end{proposition} 

Combining with Theorem \ref{gaussian}, and noting that the assumptions in Notation/Assumptions \ref{notationFM} are certainly satisfied by a sufficiently positive line bundle $L$ on the variety $X$ we get

\begin{corollary}\label{corollario} If the map $a_n$ is non-zero   then $R^{n-1}\pi_*\Q=0$.
\end{corollary}

\begin{proof} (of Proposition \ref{secondreduction})
We apply $\Ext^n_{Y\times A}(\>\cdot\> ,\OO_{\Delta_Y})$ to the usual exact sequence
\begin{equation}\label{usual} 0\rightarrow\I_{\Delta_C/Y\times C}\otimes p_2^*N^\on_{C}\rightarrow  p_2^*N^\on_{C}\rightarrow {\delta}_*N^\on_{C}\rightarrow 0
\end{equation}
 Using the spectral sequence (\ref{spec2}) and the isomorphisms provided by Prop. \ref{key} we get the commutative exact diagram
 \begin{equation}\label{fundamental}\xymatrixcolsep{1pc}\xymatrixrowsep{1.5pc}
\xymatrix{H^1(N^{-1}_{C})\otimes\Lambda^0T_{A,0}\ar[r]^=\ar[d]^-= &H^1(N^{-1}_{C})\otimes\Lambda^0T_{A,0}\ar[d]\\ 
 H^1(N^{-1}_{C})\otimes\Lambda^0T_{A,0}\ar@{^{(}->}[r]^-{v_n}&\Ext^n_{Y\times A}(p_2^*N^\on_{C},\OO_{\Delta_Y})\ar@{->>}[r]^-{w_n}\ar[d]^f&
 H^0(\N^\prime\otimes N_C^{-1})\ar@{^{(}->}[d]\\
 &\Ext^n_{Y\times A}(\I\otimes p_2^*N^\on_{C},\OO_{\Delta_Y})\ar[d]&
  H^1(T_Y\otimes N^{-1}_{C})\ar@{^{(}->}[l]^-u\ar[d]^{G^Y_A}\\
  &H^1(N^{-1}_{C})\otimes T_{A,0}\ar[r]^-=&H^1(N^{-1}_{|C})\otimes T_{A,0}\\ }
  \end{equation}
where:\\
- we have used  (\ref{prep2}) to compute  
$$\Ext^i_{Y\times A}({\delta_C}_*N^\on_{C},\OO_{\Delta_Y})\cong H^1(C,N^{-1}_{C})\otimes \Lambda^{i-n+1}T_{A}$$
-  for typographical brevity we have denoted
$$\I:=\I_{\Delta_C/Y\times C}$$ and the map 
\begin{equation}\label{u} u: H^1(\T_Y\otimes N_C^{-1})\rightarrowtail\Ext^n_{Y\times A}(\I\otimes p_2^*N^\on_{C},\OO_{\Delta_Y})
\end{equation}
is the composition of the natural inclusion 
\begin{gather*}H^1(\T_Y\otimes N_C^{-1})=H^1({\mathcal H}om(\I\otimes p_2^*N_C\, , \, \OO_{\Delta_C}))\hookrightarrow 
 \Ext^1_{Y\times C}(\I\, , \, {\delta_C}_*N_C^{-1})\cong\\
\cong  \Ext^1_{Y\times C}(\I\otimes p_2^*N_C^\on ,\EExt^{n-1}_{Y\times A}(\OO_{Y\times C},\OO_\DDelta))
\end{gather*}
 where the last isomorphism follows from Lemma \ref{key}, and the natural  map, arising in  the spectral sequence (\ref{spec2}),
\[ \Ext^1_{Y\times C}(\I_{{\Delta_C}/Y\times C}\otimes p_2^*N_C^\on ,\EExt^{n-1}_{Y\times A}(\OO_{Y\times C},\OO_\DDelta))\rightarrow \\
\rightarrow \Ext^n_{Y\times A}(\I_{\Delta_C/Y\times C}\otimes p_2^*N^\on_{C},\OO_{\Delta_Y}) \>\>.
\]

Next, we look at the Fourier-Mukai image of the central column of (\ref{fundamental}). In order to do so, we first 
apply 
 the Fourier-Mukai transform $\R\Phi$ to sequence (\ref{usual}). Then we apply 
$ \R\mathrm{Hom}_{Y\times \widehat A}(\cdot\>, \> \cR)$ and the spectral sequence on the $Y\times\widehat  A$-side, namely (\ref{spectral1}). 

We claim that applying 
 the Fourier-Mukai transform $\R\Phi$ to sequence (\ref{usual}), we get
 the exact sequence
\begin{equation}\label{ennesima} 0\rightarrow R^0\Phi(\I\otimes p_2^*N^\on_{C})\rightarrow 
R^0\Phi(p_2^*N^\on_{C})\rightarrow \nu^*(N^\on_{C})\otimes\cR
\rightarrow 0
\end{equation}
and the isomorphism
\begin{equation}\label{isoR1}
R^1\Phi(\I\otimes p_2^*N^\on_{C})\buildrel\sim\over\rightarrow R^1\Phi(p_2^*N^\on_{C})
\end{equation}
 Indeed we have that $R^i\Phi( {\delta}_*(N^\on_{C}))=\nu^*(N^\on_{C})\otimes\cR$ for $i=0$ and zero otherwise. 
 The map $R^0\Phi(p_2^*N^\on_{C})\rightarrow \nu^*(N^\on_{C})\otimes\cR$ is nothing else but the relative evaluation map  $$\pi^*\pi_*(\nu^*(N_C^{\on})\otimes \cR)\rightarrow \nu^*(N_C^{\on})\otimes\cR$$ and its surjectivity follows from the assumptions (see Notation/Assumptions \ref{notationFM}).
This proves what claimed.

 Eventually we get 
the following exact diagram, whose central column is the Fourier-Mukai transform of the central column of (\ref{fundamental}) and whose right column is (part of) the long cohomology sequence of $ \R\mathrm{Hom}_{Y\times \widehat A}(\cdot\>, \> \cR)$
 applied to the exact sequence (\ref{ennesima})
\begin{equation}\label{otherside}\xymatrixcolsep{0.7pc}\xymatrixrowsep{1.5pc}\xymatrix{&H^1(N_{C}^{\vee})\!\otimes\! H^{0}(\OO_{\widehat A})\ar[r]^=\ar[d]&H^1(N_{|C}^{\vee})\!\otimes\! H^{0}(\OO_{\widehat A})\ar[d]\\
\Ext^{n+1}( R^1\!\Phi(p_2^*N^\on_{C}\!),\!\cR)\ar[r]^-{a_n}\ar[d]^\cong&\Ext^n( \R\Phi(p_2^*N^\on_{C}\!),\!\cR)\ar[r]^-{b_n}\ar[d]^{FM_n(f)}&\Ext^n( R^0\!\Phi(p_2^*N^\on_{C}\!),\!\cR)\ar[d]\\
\Ext^{n+1}( R^1\!\Phi(p_2^*N^\on_{C}\!),\!\cR)\ar[r]^-{\alpha}&\Ext^n(\R\Phi(\I\!\otimes \!p_2^*N^\on_{C}\!),\!\cR)\ar[r]^-{\beta}\ar[d]&\Ext^n(R^0\!\Phi(\I\!\otimes\! p_2^*N^\on_{C}\!),\!\cR)\ar[d]\\
&H^1(N_{C}^{\vee})\!\otimes\! H^{1}(\OO_{\widehat A})\ar[r]^=&H^1(N_{C}^{\vee})\!\otimes\! H^{1}(\OO_{\widehat A})\\}
\end{equation}
For brevity, at the place on the left of the third row we have plugged the isomorphism (\ref{isoR1}). It follows, in particular, that the map $FM_n(f)$ induces the isomorphism of the images of $a_n $ and $\alpha$:
\begin{equation}\label{fivelemma}
FM_n(f):\mathrm{im}(a_n)\buildrel\cong\over\longrightarrow \mathrm{im}(\alpha)
\end{equation}
An easy diagram-chase in (\ref{fundamental}) and  (\ref{otherside}) proves the first part of the Proposition, namely that if the map $a_n$ is non-zero then the map $w_n\circ FM_n^{-1}\circ a_n$ is non-zero. The second part follows at once from the first one, (\ref{fivelemma}) and Proposition \ref{firstreduction}.
\end{proof}

\section{Proof of Theorem \ref{subvarieties}}\label{section6}

The  strategy of proof of Theorem \ref{subvarieties} is to see the two vertical exact sequences of diagram (\ref{diagramn})
as the first homogeneous pieces of two exact sequences of  graded modules over the exterior algebra. Namely, for each $i\ge n$ we have
\begin{equation}\label{prep}
\xymatrixcolsep{1pc}\xymatrixrowsep{1.5pc}\xymatrix{
\bigoplus_i\Ext^1_C(N^\on_{C},\Lambda^{i-n}{\N}^\prime\otimes N_C^{\on-1})\ar@{>->}[d]^{v_i} &
\bigoplus_i \Ext_{Y\times \widehat A}^{i+1}(R^1\Phi(p_2^*N^\on_{C}),\cR)\ar[d]^{a_i}\\
\bigoplus_i\Ext^{i}_{Y\times A}(p_2^*N^\on_{C},\OO_{\Delta_Y})\ar[r]_{\cong}^{FM_i}\ar@{->>}[d]^{w_i}&
\bigoplus_i\Ext^{i}_{Y\times \widehat A} (\R\Phi(p_2^*N^\on_{C}),\cR)\ar[d]^{b_i}\\
\bigoplus_i\mathrm{Hom}_C(N^\on_{C},\Lambda^{i-n+1}{\N}^\prime\otimes N_C^{\on-1})&\bigoplus_i\Ext_{Y\times \widehat A}^{i}(R^0\Phi(p_2^*N^\on_{C}),\cR)\\
}
\end{equation}
The exterior algebra acts on the left-hand side as $\Lambda^\bullet T_{A,0}\hookrightarrow \Ext^\bullet_{Y\times A}(\OO_\DDelta,\OO_\DDelta)$ (see (\ref{prep1}) and (\ref{prep2})). After the Fourier-Mukai transform, it acts on the right hand side as $\Lambda^\bullet H^1(\OO_{\widehat A})\hookrightarrow \Ext^\bullet_{Y\times \widehat A}(\cR,\cR)$.

\subsection{Computations in degree $\mathbf{d-1}$} In this subsection we will make some explicit calculations in degree $d-1$, where we have the special feature that the $\Hom$ space at the bottom of the left column is naturally isomorphic to $\Hom(N_C^\on,\omega_C)$ (Prof. \ref{key}(b)). The following Proposition shows that what we want to prove in degree $n$, namely that the map $w_n\circ FM^{-1}_n\circ a_n$ is non-zero, is true, in strong form, in degree $d-1$.
\begin{proposition}\label{summand}
The map $w_{d-1}$ has a canonical (up to scalar) section $\sigma$ and the injective map $\tau=(FM_{d-1})_{|Im(\sigma)}$ factorizes trough $a_{d-1}$. \ Summarizing, in degree $i=d-1$ diagram (\ref{prep}) specializes to
\begin{equation}\label{prepxy}
\xymatrixcolsep{1pc}\xymatrix{
\Ext^1_C(N^\on_{C},\Lambda^{d-1-n}{\N}^\prime\otimes N_C^{\on-1})\ar@{>->}[d]^{v_{d-1}} &
 \Ext_{Y\times \widehat A}^{d}(R^1\Phi(p_2^*N^\on_{C}),\cR)\ar[d]^{a_{d-1}}\\
\Ext^{d-1}_{Y\times A}(p_2^*N^\on_{C},\OO_{\Delta_Y})\ar@{<->}[r]_{\cong}^{FM_{d-1}}\ar@{->>}[d]^{w_{d-1}}
\ar@/^1pc/[ur]^\tau&
\Ext^{d-1}_{Y\times \widehat A} (\R\Phi(p_2^*N^\on_{C}),\cR)\ar[d]^{b_{d-1}}\\
\mathrm{Hom}_C(N^\on_{C},\omega_C)\ar@/ ^3pc/[u]_{\sigma}&\Ext_{Y\times \widehat A}^{d}(R^0\Phi(p_2^*N^\on_{C}),\cR)\\
}
\end{equation}
\end{proposition}
\begin{proof} The section $\sigma$ is given (up to scalar) by
 the product map
\begin{equation}\label{L} \Ext^{d-1}_{Y\times  A}(p_2^*N^\on_{C},\OO_{Y\times  A})\otimes \mathrm{Hom}_{Y\times A}(\OO_{Y\times A},\OO_{\Delta_Y})\buildrel \sigma\over\rightarrow  \Ext^{d-1}_{Y\times  A}(p_2^*N^\on_{C},\OO_{\Delta_Y})\end{equation}
In fact, note that, 
\[\Ext^{d-1}_{Y\times  A}(p_2^*N^\on_{C},\OO_{Y\times  A})
\cong p_1^*H^0(\OO_Y)\otimes p_2^*\Ext^{d-1}_A(N^\on_{C},\OO_A)
\cong  p_1^*H^0(\OO_Y)\otimes p_2^*\mathrm{Hom}_C(N^\on_{C},\omega_C)
\]
The fact that $s$ is a section of $w_{q-1}$ is clear, as the latter is the natural map 
\begin{gather*}\Ext^{d-1}_{Y\times  A}(p_2^*N^\on_{C},\OO_{\Delta_Y})\rightarrow 
H^0(\EExt^{d-1}_{Y\times  A}(p_2^*N^\on_{C}),\OO_{\Delta_Y})\cong \\
\cong 
H^0(\EExt^{d-1}_{Y\times  A}(p_2^*N^\on_{C},\OO_{Y\times A})\otimes \OO_{\Delta_Y})\cong 
\Ext^{d-1}_{Y\times  A}(p_2^*N^\on_{C},\OO_{Y\times  A})\otimes H^0(\OO_\DDelta)
\end{gather*}

Next, we prove the second part of the statement. On the $Y\times \widehat A$-side, we consider the following  product map
\begin{equation}\label{c}\xymatrixcolsep{0.7pc}\xymatrixrowsep{1.5pc}\xymatrix{
\mathrm{Hom}_{Y\!\times \!\widehat A}(R^1\Phi(p_2^*N^\on_{C}), \OO_{Y\!\times \hat 0}) \! \otimes\! 
\Ext^d_{\!Y\times \!\widehat A}( \OO_{Y\!\times\hat 0} , \!\cR)\ar[r]\ar[d]^\cong&\Ext_{Y\!\times\! \widehat A}^{d}(R^1\Phi(p_2^*N^\on_{C}),\!\cR)\ar[d]^{a_{d-1}}\\
\Ext^{-1}_{\!Y\times\!\widehat A}(\R\Phi (p_2^*N^\on_{C}), \OO_{Y\!\times \hat 0})  \!\otimes\!
\Ext^d_{Y\!\times\! \widehat A}( \OO_{Y\!\times\hat 0} , \!\cR)\ar[r]&\Ext_{Y\!\times\! \widehat A}^{d-1}(\R\Phi (p_2^*N^\on_{C}),\!\cR)}
\end{equation}
where the vertical isomorphism comes from the usual spectral sequence (\ref{spectral1}).
 By (\ref{inversion}) the inverse of the Fourier-Mukai transform is $(-1)_A^*\circ \R\Psi[q]$. By (\ref{k(0)}) we have that 
\begin{align*}(-1)_A^*\circ \R\Psi(\OO_{Y\times\hat 0})&=(-1)_A^*\circ R^0\Psi(\OO_{Y\times \hat 0})=  \OO_{Y\times A}\\
(-1)_A^*\circ \R\Psi(\cR)&= (-1)^*_A\circ R^d\Psi(\cR)[-d]=\OO_{\Delta_Y}
\end{align*}
Therefore, thanks to Mukai's inversion theorem (\ref{inversion}), the Fourier-Mukai transform identifies -- on the $Y\times A$-side -- the sources of both rows in diagram (\ref{c}) to
\[\Ext^{-1}_{Y\times A}(p_2^*N^\on_{C}[-d],\OO_{Y\times A})\otimes \Ext^d(\OO_{Y\times A}, \OO_{\Delta_Y}[-d])
\cong \Ext^{d-1}_{Y\times  A}(p_2^*N^\on_{C},\OO_{ Y\times A})\otimes \mathrm{Hom}_{Y\times A}(\OO_{Y\times A},\OO_{\Delta_Y})
\]
This concludes the proof of the Proposition.
\end{proof}

\subsection{Conclusion of the proof of Theorem \ref{subvarieties}} 
\begin{notation}\label{final}We introduce the following typographical abbreviations on diagram (\ref{prep}): the isomorphic (via the Fourier-Mukai transform) spaces of the central row of diagram (\ref{prep}) are identified to vector spaces $E_i$, and we denote $V_i,E_i, W_i$ the spaces appearing in the left column of diagram (\ref{prep}) (from top to down), and $A_i, E_i, B_i$ the spaces appearing in the right column (from top to down). We denote also $\Lambda^\bullet T_{A,0}$ the acting exterior algebra. The structure of $\Lambda^\bullet T_{A,0}$-graded modules induces a natural map of exact sequences (we focus on degrees $n$ and $d-1$ as they are the relevant ones in our argument)
\begin{equation}\label{diagrammone}
\begin{matrix}
\xymatrix{&A_n\ar[d]^{a_n}\\ V_n\ar@{>->}[r]^{v_n}&E_n\ar@{->>}[r]^{w_n} \ar[d]^{b_n}&W_n\\
&B_n\\}\\
&{}\\ \downarrow^\phi \\ 
\xymatrix{\\&\overset{d-1-n}\Lambda T^\vee_{A,0}\otimes A_{d-1}\ar[d]^{ \tilde a_{d-1}}\\\overset{d-1-n}{\Lambda}T_{A,0}^\vee\otimes V_{d-1}\ar@{>->}[r]^{\tilde v_{d-1}}&
\overset{d-1-n}\Lambda T_{A,0}^\vee\otimes E_{d-1}\ar@{->>}[r]^{\tilde w_{d-1}} \ar[d]^{\tilde b_{d-1}}&\overset{d-1-n}\Lambda T^\vee_{A,0}\otimes W_{d-1}\\
&\overset{d-1-n}\Lambda T_{A,0}^\vee\otimes B_{d-1}\\} 
\end{matrix}
\end{equation}
where we have denoted $\tilde v_{d-1}=\mathrm{id}\otimes v_{d-1}$ and so on. We denote also 
$$\phi_{A_n}:A_n\rightarrow \overset{d-1-n}\Lambda T_{A,0}^\vee\otimes A_{d-1}$$ and, similarly, $\phi_{B_n}$, \ $ \phi_{V_n}$, \ $\phi_{W_n}$, \  $\phi_{E_n}$ \ . 
\end{notation}

At this point we make the following assumption

\noindent (*) \emph{the extension class $e$ of the restricted cotangent sequence
\begin{equation}\label{normal,again}
0\rightarrow N_C^\vee\rightarrow (\Omega^1_X)_{|C}\rightarrow (\Omega^1_Y)_{|C}\rightarrow 0
\end{equation}
belongs to the subspace $H^1(\T_Y\otimes N_C^\vee)$ \ of \ $\Ext^1_C(\Omega^1_Y\otimes N_C,\OO_C)$ }\footnote{these are the locally trivial first-order deformations}.
Note that if $C$ is smooth and $Y$ is smooth along $C$ this is obvious, since the two spaces coincide.

\begin{remark}\label{e} Note that, if (*) holds then $e$ belongs to the subspace $H^0(\N^\prime \otimes N_C^\vee)$ of $H^1(\T_Y\otimes N_C^\vee)$: as mentioned in Remark \ref{mappaG}  from the exact sequence defining the restricted equisingular normal sheaf  (\ref{normal}) we get that
$$H^0(C,\N^\prime \otimes N_C^\vee)=\ker \bigl(\,\,H^1(C,\T_Y\otimes N_C^\vee)\buildrel G\over\rightarrow H^1(C,\T_A\otimes N_C^\vee)\,\,\bigr)$$
 The fact that $e$ belongs to $H^0(\N^\prime\otimes N_C^\vee)$  essentially follows from the deformation-theoretic interpretation of this map $G$ (it is the (restriction to $H^1(C,\T_Y\otimes N_C^\vee)$ of the map $G^Y_Z$ of (\ref{general}) in the Introduction, with $Z=A$). More formally: the target of $G$ is $\Hom_k(\Omega^1_{A,0}, H^1(C,N_C^\vee))$ and $G$ takes an extension class $f$ to the map $\Omega^1_{A,0}\rightarrow H^1(N_C^{-1})$ obtained by composing the coboundary map of $f$ with the map $\Omega^1_{A,0}\rightarrow H^0((\Omega^1_Y)_{|C})$. If the extension class is (\ref{normal,again}) then this map factorizes through $H^0((\Omega^1_X)_{|C})$, hence $e\in\ker G$.  
\end{remark}

From diagram (\ref{diagrammone}) we have the map
$$\phi_{W_n}: H^0(C,\N^\prime\otimes N_C^\vee)=\ker (G)\rightarrow \Hom(\overset{d-1-n}\Lambda T_{A,0}, H^0( \omega_C\otimes N_C^{\otimes - n}))$$ 
 
\begin{lemma}\label{non-zero}
$\phi_{W_n}(e)\ne 0\>\>.$
\end{lemma}
\begin{proof} We make the identification $\overset{d-1-n}\Lambda T_{A,0}\cong \overset{n+1}\Lambda \Omega^1_{A,0}$. Accordingly $\phi_{W_n}(e)$ is identified to a map
$$\phi_{W_n}(e):\overset{n+1}\Lambda\Omega^1_{A,0}\rightarrow H^0(\omega_C\otimes N_C^{\otimes - n})$$
We consider the map
\begin{equation}\label{lambda}\overset{n+1}\Lambda\Omega^1_{A,0}\rightarrow H^0((\overset{n+1}\Lambda\Omega^1_X)_{|C})
\end{equation}
obtained as $\Lambda^{n+1}$ of the co-differential $\Omega^1_{A,0}\otimes \OO_C\rightarrow (\Omega^1_X)_{|C}$. Since the co-differential is surjective the map (\ref{lambda}) is non-zero. If $C$ is smooth and $X$  and $Y$ are smooth along $C$ then the target of (\ref{lambda}) is $H^0((\omega_X)_{|C})=H^0(\omega_C\otimes N_C^{\otimes -n})$. Via the above identifications,  the map $\phi_{W_n}(e)$ coincides, up to scalar, with  (\ref{lambda}). The Lemma follows in this case.  
 Even if $X$ is not smooth along $C$ the $\phi_{W_n}(e)$ is the composition of  the map (\ref{lambda}) and the $H^0$ of the canonical map $\Lambda^{n+1}((\Omega^1_X)_{|C})\rightarrow (\omega_X)_{|C}\cong \omega_C\otimes N_C^{\otimes - n}$. Such composition is clearly non-zero and the Lemma follows as above. 
 \end{proof}
 
 At this point, the line of the argument is clear. The class $\phi_{W_n}(e)$ is non-zero, and, by Proposition \ref{summand} it belongs to ${\mathrm Im}\,(a_{d-1})\cap {\mathrm Im}\, (\phi_{E_n})$. This implies that ${\mathrm Im}\, (a_n)$ is non-zero, since otherwise $E_n$ would be isomorphic to a subspace of $B_n$ and ${\mathrm Im}\, (\phi_{E_n})$ would be contained in $\Lambda^{d-1-n}T^\vee\otimes B_{d-1}$.
 
  By Corollary \ref{corollario} this proves that $R^n\pi_*\Q=0$.

To prove the vanishing of $R^i\pi_*\Q$ for $0<i<n$ one takes a sufficiently positive ample line bundle $M$ on $X$ and a $i+1$-dimensional complete intersection of divisors in $|M|$, say $X^\prime$. It follows from relative Serre vanishing that $R^i\pi_*(\Q)=R^i\pi_*(\Q_{|X^\prime\times \widehat A})$. Therefore the desired vanishing follows by the previous step.  The vanishing of $R^0\pi_*\Q$ is standard: as it is a torsion-free sheaf, it is enough to show that its support is a proper subvariety of $\widehat A$. By base change, this is contained in the locus of $\alpha\in \widehat A$ such that $h^0(X,\alpha_{|X})>0$, i.e. the kernel of the homomorphism $\Pic0 A\rightarrow \mathrm{Pic} X$,  which is easily seen to be a proper subvariety of $\Pic0 A$ \footnote{for example, one can reduce to prove the same assertion for a general curve $C$ complete intersection of $n$ irreducible effective divisors
 in $|L|$ for a sufficiently positive line bundle $L$ on $X$}. This concludes the proof of Theorem \ref{subvarieties}. \qed

\begin{remark}\label{sernesi} The hypothesis that $X$ is smooth in codimension one is used to ensure that assumption \emph{(*)} can be made.
\end{remark}

\section*{Ackowledgements} 

Thanks to Christopher Hacon and Sandor Kovacs for valuable correspondence, and especially for showing me their counterexamples to a previous wrong statement of mine. Thanks also to Antonio Rapagnetta and Edoardo Sernesi for valuable discussions.
I also thank Hacon, Mircea Mustata and Mihnea Popa for allowing me to contribute to this volume, even if I couldn't participate to the Robfest.
Above all my gratitude goes to Rob Lazarsfeld. Most of my understanding of the matters of this paper goes back to his teaching.

\providecommand{\bysame}{\leavevmode\hbox
to3em{\hrulefill}\thinspace}


\begin{thebibliography}{13}


\bibitem{bf} C. Banica and O. Forster, {Multiplicity structures on space curves}, in \emph{The Lefschetz Centennial Conference, proceedings on Algebraic Geometry}, AMS (1984), 47--64

 \bibitem{bm} A. Beauville. and J. Y. M\'erindol, {Sections hyperplanes des surfaces K3}, Duke Math. Jour., \textbf{55} no. 4, (1987), 873--878.
 
 

\bibitem{cfp}
E. Colombo, P. Frediani and G. Pareschi, {Hyperplane sections of abelian surfaces}, J. Alg. Geom \textbf{21} (2012), 183--200





\bibitem{ferrand} D. Ferrand, {Courbes gauches et fibr\'es de rang 2}, C.R.A.S. \textbf{281} (1977) 345--347

\bibitem{gl1}
M. Green and R. Lazarsfeld, {Deformation theory, generic vanishing
theorems, and some conjectures of Enriques, Catanese and Beauville},
Invent. Math. \textbf{90} (1987), 389--407.

\bibitem{gl2}
M. Green and R. Lazarsfeld, {Higher obstructions to deforming cohomology groups of line bundles}, J. Amer. Math. Soc. \textbf{1} (1991), no.4, 87--103.




\bibitem{laz} Lazarsfeld, ~R., Brill-Noether-Petri without
degeneration, {\em J. of Differential Geom. (3)} {\bf 23} (1986),
299-307.

\bibitem{rd} R. Hartshorne, \emph{Residues and Duality}, Springer (1966)



\bibitem{hacon}
Ch. Hacon, {A derived category approach to generic vanishing}, J.
Reine Angew. Math. \textbf{575} (2004), 173--187.

\bibitem{hacko} Ch. Hacon and S. Kovacs, {Generic vanishing fails for singular varieties and in characteristic $p>0$},
preprint 	arXiv:1212.5105 [math.AG]





\bibitem{kempf}
G. Kempf: \emph{Toward the inversion of abelian integrals, I, } Ann. Math. 110 (1979), 184-202









\bibitem{eisenbud} S. Mac Lane, \emph{Homology}, Springer, 1963

\bibitem{M} D. Mumford, \emph{Abelian varieties}, 2nd ed., Oxford University press, London, 1974

\bibitem{mu1}
S. Mukai, {Duality between $D(X)$ and $D(\widehat{X})$ with its
application to Picard sheaves}, Nagoya Math. J. \textbf{81} (1981),
153--175.


\bibitem{mu2}
S. Mukai, {Fourier functor and its application to the moduli of bundles
on an abelian variety}, In: Algebraic Geometry, Sendai 1985, Advanced studies
in pure mathematics \textbf{10} (1987), 515--550.

\bibitem{pp1} G. Pareschi and M. Popa, 
{Strong generic vanishing and a higher dimensional Castelnuovo-de Franchis inequality}, Duke Mathematical Journal, \textbf{150} (2009), 269--285. 

\bibitem{pp2} G. Pareschi and M. Popa, {GV-sheaves, Fourier-Mukai transform, and Generic Vanishing}, Amer. J. of Math. \textbf{133} (2011), 235--271

\bibitem{pp3} G. Pareschi and M. Popa, {Regularity on Abelian Varieties III: Relationship with Generic Vanishing and Applications} in {\em Grassmannians, Moduli Spaces and Vector Bundles}, AMS (2011) 141--167

\bibitem{sernesi} E. Sernesi, \emph{Deformations of algebraic schemes}, Springer (2006)

\bibitem{v} Voisin,~C., Sur l'application de Wahl des courbes satisfaisant la condition de Brill-Noether-Petri, {\em Acta Math.} {\bf 168} (1992), 249--272.

\bibitem{wahl} Wahl,~J., Introduction to Gaussian maps on an algebraic curve,
  {\em Complex projective geometry},
London Math. Soc. Lecture Note Ser., 179, Cambridge Univ. Press,
Cambridge, (1992), 304--323.


\end{thebibliography}
\end{document}